\documentclass{amsart}
\author{Julien Melleray}
\address{Universit\'e Claude Bernard -- Lyon 1 \\
  Institut Camille Jordan, CNRS UMR 5208 \\
  43 boulevard du 11 novembre 1918 \\
  69622 Villeurbanne Cedex \\
  France}
\usepackage{amssymb}
\usepackage[initials,shortalphabetic]{amsrefs}
\usepackage{url}
\usepackage[all]{xy}
\usepackage[autostyle]{csquotes}
\usepackage{newpxtext,newpxmath}
\usepackage{hyperref}
\usepackage{my-macros}
\usepackage{graphicx}
\usepackage{stmaryrd}
\usepackage{appendix}
\numberwithin{equation}{section}

\title[A dynamical proof of Matui's absorption theorem]{A dynamical proof of Matui's absorption theorem}

\newcounter{dummy}
\makeatletter
\newcommand\myitem[1][]{\item[#1]\refstepcounter{dummy}\def\@currentlabel{#1}}
\makeatother
\begin{document}

\begin{abstract}
We give a dynamical, relatively elementary proof of an ``absorption theorem'' which is closely related to a well-known result due to Matui. The construction is in the spirit of an earlier joint work of the author and S. Robert. In an appendix  we explain how to use this result to correct the dynamical proof given by Melleray--Robert of a classification theorem for orbit equivalence of minimal ample groups due to Giordano, Putnam and Skau (the original argument had a gap).
\end{abstract}
\maketitle

\section{Introduction}
This article is a continuation of \cite{Melleray2023}, and the motivation for this work is an error in one of the main arguments in that paper. We are concerned with actions of countable groups by homeomorphisms on the Cantor space $X$ which are \emph{minimal}, i.e.~such that all orbits are dense. Given an action $\Gamma \actson X$, we denote by $R_\Gamma$ the associated equivalence relation and say that two equivalence relations $R,S$ on $X$ are \emph{orbit equivalent} if there exists a homeomorphism $h \colon X \to X$ such that $(h \times h)R=S$.

Denote by $M(\Gamma)$ the set of $\Gamma$-invariant Borel probability measures on $X$.
It is not hard to see that if $h$ realizes an orbit equivalence between the relations induced by actions of two countable groups $\Gamma$, $\Lambda$ on $X$ then one must have $h_*M(\Gamma)=M(\Lambda)$. For minimal actions of $\Z$ there is a converse to that statement, proved by Giordano, Putnam and Skau \cite{Giordano1995}: if two minimal actions of $\Z$ preserve the same Borel probability measures then they are orbit equivalent. This theorem is far from being obvious; for instance, preserving the same Borel probability measures certainly does not imply that both actions have the same orbits. 

Since \cite{Giordano1995}, this classification theorem has been re-proved in several papers, including \cite{Putnam2010} and \cite{Keane2011}; all known proofs are fairly technical. Motivated in part by the perspective of extending this classification theorem for other group actions (so far it is known to hold for free actions of $\Z^d$ for any $d$, see \cite{Giordano2010}), the author and S. Robert claimed to give in \cite{Melleray2023} a purely dynamical proof of the classification theorem. I discovered recently that this proof has a gap.

Loosely speaking, the proof strategy in \cite{Melleray2023}, similarly to what is done in \cite{Giordano2004} or \cite{Putnam2010}, is to first prove that a ``small extension'' of an equivalence relation induced by a minimal $\Z$-action results in a relation which is orbit equivalent to the relation one started with; then to prove that given two equivalence relations $R$, $S$ induced by minimal $\Z$-actions which preserve the same Borel probability measures, one can produce a third equivalence relation $T$ which is a small extension of both $R$ and $S$, thereby deriving that $R$ and $S$ are orbit equivalent.

A theorem stating that a ``small extension'' of a given equivalence relation $R$ is orbit equivalent to $R$ is called an \emph{absorption theorem}. The first of those was given in \cite{Giordano2004}, then it was improved in \cite{Giordano2008} and the strongest such theorem, due to Matui, appeared in \cite{Matui2008}. Matui's proof builds on the proofs of earlier absorption theorems, making for a fairly involved argument where the dynamical aspects are complicated to understand.

In \cite{Melleray2023}, instead of working with minimal $\Z$-actions, we followed ideas of Giordano, Putnam and Skau and worked instead with minimal actions of certain locally finite subgroups of $\Homeo(X)$, which were called \emph{ample groups} by Krieger \cite{Krieger1979}. As pointed out by Putnam \cite{Putnam2010}, it seems more natural to first establish a classification theorem for minimal actions of ample groups and then derive the theorem for minimal $\Z$-actions. 

S. Robert and the author claimed in \cite{Melleray2023} to provide a dynamical proof of an absorption theorem that, while weaker than Matui's, was sufficient to prove the classification theorem of Giordano, Putnam and Skau, via an argument that only involved cutting-and-stacking methods. There is, however, a gap in that argument; once that gap is identified, it becomes apparent that for that approach to work one needs a stronger absorption theorem, and that Matui's absorption theorem is adequate to the task. We give here an elementary proof of such a strong absorption theorem (at the end of the paper, we briefly sketch how one can recover Matui's absorption theorem from our main result, Theorem \ref{t:absorption}). 

To study ample group actions up to orbit equivalence, it is common to employ Bratteli diagrams, which have been instrumental in the proofs of several deep results (see e.g.~\cite{Giordano1995}, \cite{Giordano2004}, \cite{Giordano2008}, \cite{Matui2008}). While those diagrams are very natural for someone with a background in operator algebras (and are nicely connected with homological invariants) and their effectiveness to tackle the type of questions we are concerned with is well established, their use can lead to proofs where the dynamical aspects are hard to grasp. Here, as in \cite{Melleray2023}, we always work directly with clopen subsets of $X$ via cutting-and-stacking methods. Still, it must be pointed out that many of the ideas and concepts that we use are closely related to those found in the works of Giordano, Matui, Putnam and Skau mentioned above. 

Let us briefly discuss the organization of the paper. After reviewing some basic notions needed for our argument, we first develop some elementary theory of what we call malleable subsets. These are analogs of the étale extensions considered in \cite{Giordano2004}, \cite{Giordano2008} and \cite{Matui2008} and provide the paradigm for the ``small extensions'' alluded to above. We then need to develop some machinery in order to prove the absorption theorem. To that end, we
extend a theorem of Krieger \cite{Krieger1979}; a consequence of that work is a homogeneity result which is instrumental in our proof of the absorption theorem (see Lemma \ref{l:any_copy_will_do}). Using this theorem of Krieger as a step towards the classification theorem is one of the key ideas of \cite{Melleray2023} and our strategy here is similar.
 Then we prove our version of Matui's absorption theorem (Theorem~\ref{t:absorption}), using a method which is related to what Putnam calls the ``Hilbert--Bratteli hotel'' in \cite{Putnam2018}. Informally, to prove that the relation $R_\Gamma$ induced by the action of a minimal ample group $\Gamma$ is orbit equivalent to a small extension of itself, we begin by showing (see Lemma \ref{l:everyone_is_a_small_extension} and how it is used to prove Theorem \ref{t:absorption}) that $R_\Gamma$ can be obtained from the relation $R_\Lambda$ induced by another minimal ample group $\Lambda$ by repeating countably many times the same small extension. Thus one more small extension should not (and, as it turns out, does not) change the orbit equivalence class of $R_\Gamma$. The absorption theorem may thus be thought of as an analogue of the classical ordinal equation $1+ \omega=\omega$ (hence the analogy with the Hilbert hotel). After that we briefly sketch how to recover Matui's absorption theorem from our main result.

At the end of the paper the reader will find an appendix, which serves as a corrigendum to \cite{Melleray2023}. There I assume that the reader is familiar with the arguments and notations of \cite{Melleray2023} and explain how use the improved absorption theorem so as to fix the proof of the classification theorem for orbit equivalence given in \cite{Melleray2023}.

\medskip
\emph{Acknowledgements.} I am very grateful to F. Le Maître for making many useful comments and suggestions after reading a first version of this article. I am also indebted to an anonymous referee for thoughtful advice.

\section{Background and vocabulary}
We recall some notions and terminology. The interested reader can find a much more detailed exposition in \cite{Melleray2023}, with proofs for some statements we only mention here in passing.

Assume that $X$ is a compact, metrizable, $0$-dimensional space. Given a subgroup $\Gamma \le \Homeo(X)$, we denote by $M(\Gamma)$ the set of all Borel probability measures on $X$ which are $\Gamma$-invariant. The \emph{full group} $F(\Gamma)$ generated by $\Gamma$ is the set of all $g \in \Homeo(X)$ such that there exists a clopen partition $X= \bigsqcup_{i=1}^n U_i$ and $\gamma_1,\ldots,\gamma_n \in \Gamma$ such that $g(x)=\gamma_i(x)$ for all $x \in U_i$. We say that $\Gamma$ is a full group if $\Gamma=F(\Gamma)$. Note that we always have $M(\Gamma)=M(F(\Gamma))$.

For any countable subgroup $\Gamma \le \Homeo(X)$, we denote by $R_\Gamma$ the associated equivalence relation on $X$, and let $[R_\Gamma]$ denote the subgroup of all $g \in \Homeo(X)$ which map each $\Gamma$-orbit to itself. Then $[R_\Gamma]$ is a full group and $M([R_\Gamma])=M(\Gamma)$.

An \emph{ample group over $X$} is a countable, locally finite full group $\Gamma \le \Homeo(X)$ with the property that for all $\gamma \in \Gamma$ the set $\{ x \in X : \gamma(x)=x \}$ is clopen in $X$.

Given a subgroup $\Gamma \le \Homeo(X)$ and a Boolean subalgebra $\mcA$ of $\Clopen(X)$, we say that $(\mcA,\Gamma)$ is a \emph{unit system}\footnote{There is (again) an imprecision in \cite{Melleray2023}: the definition given there is different from this one for unit systems which are not finite. This is not an issue since those are not used anywhere in \cite{Melleray2023}.} if:
\begin{itemize}
\item For every $A \in \mcA$ and every $\gamma \in \Gamma$, $\gamma(A) \in \mcA$, giving us an evaluation map $e_\mcA \colon \Gamma \to \mathrm{Aut}(\mcA)$.
\item The morphism $e_\mcA$ is injective (equivalently, the only element of $\Gamma$ mapping every element of $\mcA$ to itself is the identity).
\item For every $\gamma \in \Gamma$, $\{x : \gamma(x)= x \} \in \mcA$.
\item For every $g \in \Homeo(X)$, if there exists a partition $X=\bigsqcup_{i=1}^n A_i$ with $A_i \in \mcA$ such that $g$ coincides on each $A_i$ with some $\gamma_i \in \Gamma$, then $g \in \Gamma$.
\end{itemize}
We sometimes denote by $\Gamma_\mcA$ the subgroup $e_\mcA(\Gamma)$ of $\mathrm{Aut}(\mcA)$. 

One says that the unit system $(\mcA,\Gamma)$ is \emph{finite} if $\mcA$ is finite, in which case $\Gamma$ is also finite. We say that a unit system $(\mcB,\Sigma)$ \emph{refines} another unit system $(\mcA,\Gamma)$ if $\mcB$ contains $\mcA$ and $\Sigma$ contains $\Gamma$.

Krieger \cite{Krieger1979}*{Lemma~2.1} proved that for any ample group there exists a sequence $(\mcA_n,\Gamma_n)_n$ of finite unit systems such that $(\mcA_{n+1},\Gamma_{n+1})$ refines $(\mcA_n,\Gamma_n)$ for all $n$ and
\[\bigcup_n \mcA_n=\Clopen(X) \quad ; \quad \bigcup_n \Gamma_n=\Gamma .\] 
We say that $(\mcA_n,\Gamma_n)$ is an \emph{exhaustive sequence} for $(X,\Gamma)$.

\textbf{From now on (throughout the paper) the letter $X$ stands for the Cantor space}.

An action $\Gamma \actson X$ by homeomorphisms is \emph{minimal} if all of its orbits are dense; we say that $\varphi \in \Homeo(X)$ is minimal if the $\Z$-action $n \cdot x=\varphi^n(x)$ is minimal. We denote by $F(\varphi)$ the full group generated by $\{\varphi^n : n \in \Z\}$ (it is often denoted $\llbracket \varphi \rrbracket$ in the literature). This is a countable group which acts minimally on $X$; for any $x_0 \in X$, the subgroup 
\[F_{x_0}(\varphi)= \{ \gamma \in F(\varphi) : \gamma(O^+(x_0))=O^+(x_0)\}\]
is an ample group (here $O^+(x_0)=\{ \varphi^n(x_0) : n \ge 0\}$). The orbits of $F_{x_0}(\varphi) \actson X$ are the same as the $\varphi$-orbits, except for the orbit $O(x_0)$ which splits into its positive part $O^+(x_0)$ and negative part $O^-(x_0)$. It is proved in \cite{Melleray2023} that all minimal ample groups over $X$ can be realized in this way, so minimal ample groups and topological full groups of minimal $\Z$-actions are closely related.

Any action $\Gamma \actson X$ also induces an equivalence relation $\sim_\Gamma$ on $\Clopen(X)$, where $A \sim_\Gamma B$ iff there exists $\gamma \in \Gamma$ such that $\gamma(A)=B$. If $\Gamma$ is ample, then $\sim_\Gamma$ is \emph{full} in the following sense: whenever $A, B \in \Clopen(X)$ are such that $A= \bigsqcup_{i=1}^n A_i$, $B=\bigsqcup_{i=1}^n B_i$ and $A_i \sim_\Gamma B_i$ for all $i$ then $A \sim_\Gamma B$.



\begin{defn}
We say that two ample subgroups $\Gamma$, $\Lambda$ \emph{induce isomorphic relations on $\Clopen(X)$} if there exists $h \in \Homeo(X)$ such that for any clopen $A,B$ one has $(A \sim_\Gamma B) \Leftrightarrow (h(A) \sim_\Lambda h(B))$.
\end{defn}

Krieger's theorem alluded to in the introduction (and which we strengthen in Section \ref{s:Krieger}) implies that if two ample groups $\Gamma$, $\Lambda$ induce isomorphic relations on $\Clopen(X)$ then $\Gamma$ and $\Lambda$ are conjugated in $\Homeo(X)$. In particular $R_\Gamma$ and $R_\Lambda$ are then orbit equivalent. 

The following lemma, whose analogue for $\Z$-actions is due to Glasner and Weiss \cite{Glasner1995a}, is crucial for our approach.


\begin{lemma}[Glasner--Weiss]\label{l:Glasner_Weiss}
Assume that $\Gamma \le \Homeo(X)$ is an ample group. Fix $A,B \in \Clopen(X)$. 
If $\mu(A) < \mu(B)$ for all $\mu \in M(\Gamma)$ then there exists $\gamma \in \Gamma$ such that $\gamma(A)\subset B$.
\end{lemma}

We will use another property of $\sim_\Gamma$ (which is implicitly used in \cite{Melleray2023}).

\begin{lemma}\label{l:equidecomposability_on_clopens_is_induced_by_full_group}
Let $\Gamma$ be an ample group over $X$, and $A$, $B$ be clopen subsets of $X$.

Assume that $A \sim_\Gamma B$ and that $A_1 \subseteq A$, $B_1 \subseteq B$ are clopen sets such that $A_1 \sim_\Gamma B_1$. Then $A \setminus A_1 \sim_\Gamma B \setminus B_1$.
\end{lemma}

\begin{proof}
Fix an exhaustive sequence $(\mcA_n,\Gamma_n)$ of finite unit systems for $\Gamma$. 
By assumption, there exists $n$ and $\gamma$, $\delta \in \Gamma_n$ such that $A$, $A_1$, $B$, $B_1$ all belong to $\mcA_n$, $\gamma A=B$ and $\delta A_1=B_1$.

Given an atom $U$ of $\mcA_n$ and $C \in \mcA_n$ define
\[n_C(U)= \left|\{V \in \Gamma_n U : V \subseteq C\}\right| .\]

Since $A,B \in \mcA_n$, the existence of $\gamma \in \Gamma_n$ such that $\gamma A=B$ amounts to saying that $n_A(U)=n_B(U)$ for every atom $U$ of $\mcA_n$; similarly, we obtain $n_{A_1}(U)= n_{B_1}(U)$ for every $U$.
Then for every $U$ we have 
\[ n_{A \setminus A_1}(U)=n_A(U) - n_{A_1}(U) = n_B(U)-n_{B_1}(U) = n_{B \setminus B_1}(U) \]
and that proves the lemma.
\end{proof}

\section{Malleable subsets}
Throughout this section we fix an ample $\Gamma$ group over $X$ which acts minimally. 
Below, we will use without mention the fact that for any closed subset $K$ of $X$, any clopen subset of $K$ is the intersection of $K$ with a clopen subset of $X$.

\begin{defn}[see \cite{Giordano2004}*{Definition~4.11}]
Let $K$ be a closed subset of $X$. We say that $K$ is $\Gamma$-\emph{thin} if $\mu(K)=0$ for every $\mu \in M(\Gamma)$.
\end{defn}

The next lemma is standard.
\begin{lemma}\label{l:small_clopen}
Assume that $K$ is a closed, $\Gamma$-thin subset of $X$. Then for every $\varepsilon >0$ there exists a clopen subset $U$ such that $K \subset U$ and $\mu(U) \le \varepsilon$ for all $\mu \in M(\Gamma)$.
\end{lemma}

\begin{proof}
Pick a decreasing sequence $(U_n)_{n \in \N}$ of clopen subsets such that $\bigcap_n U_n = K$. Given $n \in \N$, the function $\Phi_n \colon \mu \mapsto \mu(U_n)$ is continuous on $M(\Gamma)$ (endowed with its usual compact topology). 

Furthermore, for every $\mu \in M(\Gamma)$ the sequence $(\Phi_n(\mu))_n$ decreases to $\mu(K)=0$. Since $M(\Gamma)$ is compact we can apply Dini's theorem and conclude that there is $n$ such that $\Phi_n(\mu) \le \varepsilon$ for all $\mu \in M(\Gamma)$, equivalently $\mu(U_n) \le \varepsilon$ for all $\mu \in M(\Gamma)$.
\end{proof}

Then the argument used to prove \cite{Melleray2023}*{Lemma 3.13} gives the following result.

\begin{lemma}\label{l:compatible_clopen}
Assume that $K$ is a closed $\Gamma$-thin subset. Let $U$ be a nontrivial clopen subset of $X$, and $A$ a clopen subset of $K$. Let $V \in \Clopen(X)$ be such that $A \subset V \cap K$ and $\mu(U)< \mu(V)$ for all $\mu \in M(\Gamma)$.

Then there exists $U' \in \Clopen(X)$ such that $U' \cap K =A$, $U' \subset V$, and $U' \sim_\Gamma U$.
\end{lemma}

\begin{proof}
By compactness of $M(\Gamma)$ (and minimality of the action) there exists $\eta >0$ such that both $\eta < \mu(U)$ and $\mu(U)+ \eta < \mu(V)$ for all $\mu \in M(\Gamma)$.

By Lemma \ref{l:small_clopen}, there exists a clopen subset $W$ containing $K$ such that $\mu(W)< \eta$ for all $\mu \in M(\Gamma)$. Since $A$ is clopen in $K$, we have $A=K \cap B$ with $B$ clopen in $X$ and contained in $W \cap V$. Since $\mu(B) < \eta < \mu(U)$ for all $\mu \in M(\Gamma)$, we may apply Lemma \ref{l:Glasner_Weiss} and conclude that there exists a clopen $C \subseteq U$ such that $C \sim_\Gamma B$.

For all $\mu \in M(\Gamma)$ we have 
\[\mu(U \setminus C) +\eta = \mu(U) - \mu(C) + \eta < \mu(V)-\mu(B) = \mu(V \setminus B)\]
Applying Lemma \ref{l:small_clopen} again, we pick a clopen subset $D$ of $V \setminus B$ containing $(K \cap V) \setminus A$ and such that $\mu(D)< \eta$ for all $\mu \in M(\Gamma)$. We then have $\mu(U \setminus C) < \mu(V \setminus (B \cup D))$ so by Lemma \ref{l:Glasner_Weiss} there exists a clopen $E \subset V \setminus ( B \cup D)$ and such that $E \sim_\Gamma U\setminus C$. 

We conclude by setting $U'=B \cup E$.
\end{proof}

The way we used Lemma \ref{l:Glasner_Weiss} above is typical of our arguments; from now on we will not mention this lemma explicitly but it is very often in the background.

\begin{defn}
Let $K$ be a closed subset of $X$. We say that $K$ is $\Gamma$-\emph{étale} if for any $\gamma \in \Gamma$ and any clopen subset $A$ of $K$ the set $\gamma A \cap K$ is clopen in $K$. 

We say that $K$ is $\Gamma$-\emph{malleable} if it is a closed, $\Gamma$-thin and $\Gamma$-étale subset of $X$.
\end{defn}
The terminology ``étale'' comes from work of Giordano--Putnam--Skau which itself originates from operator algebra theory. Our definition amounts to saying that the restriction of $R_\Gamma$ to $K$ is étale for the topology induced by the topology of $R_\Gamma$, that is, $K$ is $R_\Gamma$-étale; see \cite{Giordano2004}*{Definition~2.1}. In \cite{Giordano2004}, one requires compatibility between $K$ and the topology of the étale equivalence relation under consideration; here, analogously, we ask for compatibility with the acting ample group. Our malleable sets correspond to the $R$-closed, $R$-étale and $R$-thin subsets considered in \cite{Giordano2004} and later in \cite{Giordano2008} and \cite{Matui2008}.

Assume that $K$ is $\Gamma$-étale. Given an involution $\gamma \in \Gamma$, define $\gamma_K \colon K \to K$ by setting $\gamma_K(x)= \gamma(x)$ if $\gamma(x) \in K$, $\gamma_K(x)=x$ otherwise. 
Denote $K_1= \{x \in K : \gamma(x) \in K\}$. Then $K_1$ is clopen in $K$ since $K$ is $\Gamma$-étale, and $\gamma_K(K_1)= K_1$. From this we obtain that $\gamma_K$ is a homeomorphic involution of $K$. 

\begin{defn}
Let $K$ be a $\Gamma$-étale subset of $X$, and $\gamma \in \Gamma$ an involution. We say that $\gamma$ is $K$-\emph{compatible} if $\gamma(K)=K$.
\end{defn}

\begin{lemma}
Let $K$ be a $\Gamma$-étale subset of $X$, and $\gamma \in \Gamma$ be an involution. Then there exists a $K$-compatible involution $\delta \in \Gamma$ such that $\delta_K=\gamma_K$.
\end{lemma}

\begin{proof}
The sets $\{x \in K : \gamma(x) \in K \}$ and $\{x \in K : \gamma(x) \ne x\}$ are both clopen in $K$; denote by $L$ their intersection. Since $\gamma$ is an involution we have $\gamma(L)=L$. 

Choose a clopen subset $U$ of $X$ which contains $L$ and is disjoint from $K\setminus L$, and set $V=U \cap \gamma(U)$. Then $V$ is clopen, $V \cap K=L$ and $V= \gamma (V)$. We can then define $\delta \in \Gamma$ by setting $\delta(x)=\gamma(x)$ for all $x \in V$ and $\delta(x)=x$ for all $x \not \in V$. 

By construction, $\delta(x)=\gamma(x)$ for all $x \in L$, and $\delta(x)=x$ for all $x \in K \setminus L$, so that $\delta_K=\gamma_K$.
\end{proof}

\begin{defn}
Let $\Gamma$ be a minimal ample subgroup of $X$ and $K$ a $\Gamma$-étale subset. Let $(\mcB,\Delta)$ be a finite unit system with $\Delta \le \Gamma$.

We say that $(\mcB,\Delta)$ is \emph{$K$-compatible} if for all atoms $A$, $B$ of $\mcB$ and every $\delta \in \Delta$ such that $\delta(A)=B$, if $K \cap A$ and $K \cap B$ are both nonempty then we have $\delta(K \cap A)= K \cap B$.
\end{defn}

This extends a definition given in \cite{Melleray2023}*{Definition~3.10}, where we only considered closed sets $K$ such that for all $x\ne y \in K$ one has $y \not \in \Gamma x$, which we called $\Gamma$-sparse sets. The assumption of $K$-compatibility amounts to saying that, if $A,B$ are two elements of the same $\Delta$-orbit that both intersect $K$, then the involution in $\Delta$ which maps $A$ to $B$ and is equal to the identity outside of $A \cup B$ is $K$-compatible.

\begin{lemma}\label{l:refining_to_compatible}
Assume that $K$ is $\Gamma$-étale and let $(\mcA,\Delta)$ be a finite unit system with $\Delta \le \Gamma$. 
Then there exists $\mcA'$ refining $\mcA$ and $\Delta'$ such that $(\mcA',\Delta')$ is a $K$-compatible finite unit system and $\Delta \le \Delta' \le \Gamma$.
\end{lemma}

Since any two finite unit systems $(\mcA,\Delta)$ and $(\mcB,\Lambda)$ with $\Delta, \Lambda \le \Gamma$ have a common refinement $(\mcC,\Sigma)$ with $\Sigma \le \Gamma$, it follows from Lemma \ref{l:refining_to_compatible} and the existence of an exhaustive sequence of finite unit systems for $\Gamma$ that if $K$ is $\Gamma$-étale then there exists an exhaustive sequence of $K$-compatible finite unit systems for $\Gamma$.

\begin{proof}
The argument proceeds by cutting the $\Delta$-orbit of each atom of $\mcA$. First, fix some such orbit $\tau$.

Let $U \ne V$ be two elements of $\tau$ and $\delta \in \Delta$ the involution such that $\delta U=V$ and $\delta(x)=x$ for every $x \not \in U \sqcup V$. We first consider $K_1(\delta) = \{x \in K \cap U : \delta(x) \in K\}$, which is clopen in $K$; we pick a clopen set $U' \subseteq U$ such that $U' \cap K=K_1(\delta)$. Then $\delta(K_1(\delta))=\delta(U' \cap K)$ is clopen in $K$ and contained in $V$, so we can find a clopen subset $V'$ of $V$ such that $V' \cap K= \delta(K_1(\delta))$. Set $U_1 = U' \cap \delta V'$, $V_1= \delta (U_1)$.

By definition, $K_1(\delta)= U_1 \cap K$ and $\delta(U_1 \cap K)= \delta(K_1(\delta))=V_1 \cap K$.

If $(U \setminus U_1) \cap K= \emptyset$ or $(V \setminus V_1) \cap K= \emptyset$ we let $U_2=U \setminus U_1$, $V_2=\delta(U_2)$ and stop partitioning $U$, $V$. Else, let $U''=U \setminus U_1$, $V''=V \setminus V_1$. Then $\delta(U'' \cap K) \cap K=\emptyset$. Since these two sets are closed, we can find a clopen set $U_2 \subset U''$ containing $U'' \cap K$ and such that $\delta(U_2) \cap K= \emptyset$, and set $V_2= \delta(U_2)$. Finally we define $U_3= U \setminus (U_1 \sqcup U_2)$ and $V_3=\delta(U_3)$.

We now have $\delta(U_1 \cap K)= V_1 \cap K$; $V_2=\delta(U_2)$ and $V_2 \cap K=\emptyset$; $V_3=\delta(U_3)$ and $U_3 \cap K=\emptyset$.

We do this for every pair of distinct $U,V$ in $\tau$ and choose a finite $\Delta$-invariant partition of $\tau$ refining all those partitions. We obtained our desired refinement of $\tau$. Doing this for every $\tau$ gives us $\mcA'$. We then let $\Delta'$ be the group of all homeomorphisms of $X$ which preserve $\mcA'$ and coincide on each atom $U$ of $\mcA$ with some $\delta_U \in \Delta$.
\end{proof}

Given a $\Gamma$-étale $K \subset X$, we denote by $\Gamma_K$ the smallest full subgroup of $\Homeo(K)$ which contains $\gamma_K$ for every involution $\gamma \in \Gamma$. For every $x \in X$ and every $\gamma \in \Gamma$ there exists an involution $\delta \in \Gamma$ such that $\delta(x)=\gamma(x)$; hence the restriction of $R_\Gamma$ to some malleable $K$ is induced by $\Gamma_K$.

The group $\Gamma_K$ can equivalently be described as follows.

\begin{lemma}\label{l:induced_group_is_what_you_expect}
Assume that $K$ is $\Gamma$-étale. Then $\Gamma_K =\{\gamma_{|K} : \gamma \in \Gamma \textrm{ and } \gamma(K)=K\}$. In particular, $\Gamma_K$ is an ample group over $K$.
\end{lemma}

\begin{proof}
Fix an exhaustive sequence $(\mcA_n,\Gamma_n)$ of $K$-compatible finite unit systems for $\Gamma$ (see the remark following the statement of Lemma \ref{l:refining_to_compatible}).

Pick $g \in \Gamma_K$. There exists a clopen partition $(U_i)_{i=1,\ldots,p}$ of $K$ such that for each $i$ the restriction of $g$ to $U_i$ coincides on $U_i$ with a product of $K$-compatible involutions. Let $A$ be a finite set of $K$-compatible involutions witnessing that fact, and define $\Delta=\langle A\rangle$ (a finite subgroup of $\Gamma$); note that $\delta(K)=K$ for all $\delta \in \Delta$, in particular restrictions of elements of $\Delta$ to $K$ form a finite subgroup of $\Homeo(K)$. 

Next, pick a finite subalgebra $\mcB$ of $\Clopen(K)$ containing every $U_i$ and such that $\delta(B) \in \mcB$ for all $\delta \in \Delta$.
Choose a Boolean subalgebra $\tilde \mcB$ of $\Clopen(X)$ such that $\{ \tilde B \cap K : \tilde B \in \tilde \mcB \}=\mcB$ and $\tilde \mcB$ is $\Delta$-invariant. For $n$ large enough we have that $\mcA_n$ refines $\tilde \mcB$ and $\Gamma_n$ contains $\Delta$. For every atom $U$ of $\mcA_n$ which intersects $K$, there exists $\delta \in \Delta$ such that $\delta$ and $g$ coincide on $U \cap K$. We must then have $\delta(U \cap K)= g(U \cap K)$, and any two such $\delta$ coincide on $U$ since $\Delta$ is contained in $\Gamma_n$ and $(\mcA_n,\Gamma_n)$ is a unit system. 

Denoting this $\delta$ by $\delta_U$, we define $\gamma \in \Homeo(X)$ by setting $\gamma(x)=\delta_U(x)$ for every atom $U$ of $\mcA_n$ which intersects $K$ and every $x \in U$, $\gamma(x)=x$ elsewhere. By definition we have both that $\gamma \in \Gamma$ and $\gamma_{|K}=g$.

Conversely, assume that $\gamma \in \Gamma$ is such that $\gamma(K)=K$. Then find $n \in \N$ such that $\gamma \in \Gamma_n$. Since $(\mcA_n,\Gamma_n)$ is $K$-compatible, for each atom $U$ of $\mcA_n$ such that $U \cap K \ne \emptyset$ there exists a unique involution $\gamma_U$ in $\Gamma_n$ such that $\gamma_U(U \cap K)= \gamma(U \cap K)$ and $\gamma_U(x)=x$ for all $x \not \in U \cup \gamma(U)$. By definition of a unit system, $\gamma$ and $\gamma_U$ coincide on $U$, and $\gamma_U$ is $K$-compatible. Let $\lambda$ be the homeomorphism of $K$ which coincides with $\gamma_U$ on each $U \cap K$ whenever $U \cap K \ne \emptyset$ and coincides with the identity everywhere else. Then $\lambda \in \Gamma_K$ and $\lambda=\gamma_{|K}$.
\end{proof}

\begin{defn}
Assume that $K$ is $\Gamma$-étale and that  $(\mcA,\Lambda)$ is a $K$-compatible finite unit system with $\Lambda \le \Gamma$. We let $\mcA_K$ denote the Boolean subalgebra of $\Clopen(K)$ induced by $\mcA$, and $\Lambda_K=\{\lambda_{|K} : \lambda\in \Lambda \text{ and } \lambda(K)=K\}$. 

Note that by $K$-compatibility we have that $(\mcA_K,\Lambda_K)$ is a finite unit system in $K$.
\end{defn}

\begin{lemma}
Assume that $(\mcA_n,\Gamma_n)$ is an exhaustive sequence of $K$-compatible finite unit systems for $\Gamma$. Then $(\mcA_{n,K},\Gamma_{n,K})$ is an exhaustive sequence of finite unit systems for $\Gamma_K$.
\end{lemma}

\begin{proof}
Clearly $\bigcup_n \mcA_{n,K}=\Clopen(K)$.
Given $\sigma \in \Gamma_K$, by Lemma \ref{l:induced_group_is_what_you_expect} there exists $\gamma \in \Gamma$ such that $\sigma=\gamma_{|K}$, whence $\sigma \in \Gamma_{n,K}$ for $n$ large enough.
\end{proof}

\section{A further strengthening of Krieger's theorem}\label{s:Krieger}

Our aim now is to establish the following version of Krieger's theorem. This is a close cousin of the ``Fundamental Lemma'' \cite{Giordano2004}*{Lemma~4.15} (each statement follows readily from the other, though the proofs are completely different) and a further strengthening of the version established in \cite{Melleray2023}*{Theorem~3.11}

\begin{theorem}\label{t:ultimate_Krieger}
Let $\Gamma$, $\Lambda$ be two minimal ample groups over $X$, and $K$ (resp. $L$) be a $\Gamma$-malleable (resp. $\Lambda$-malleable) subset of $X$. Assume also that $\Gamma$, $\Lambda$ induce isomorphic relations on $\Clopen(X)$.

Then every homeomorphism $h \colon K \to L$ such that $h\Gamma_K h^{-1} = \Lambda_L$ extends to a homeomorphism of $X$ such that $h \Gamma h^{-1}= \Lambda$.
\end{theorem}

For the remainder of this section, we fix two minimal ample groups $\Gamma$, $\Lambda$ which induce isomorphic relations on $\Clopen(X)$; by conjugating $\Lambda$ if necessary, we reduce to the case where $\sim_\Gamma$ and $\sim_\Lambda$ coincide, and we denote this relation by $\sim$. Explicitly, for any two clopen subsets $U,V$ of $X$ we have
\[ \left(\exists \gamma \in \Gamma \ \gamma U=V \right) \Leftrightarrow (U \sim V) \Leftrightarrow \left(\exists \lambda \in \Lambda \ \lambda U=V \right) .\]

We also fix a $\Gamma$-malleable subset $K$ of $X$, a $\Lambda$-malleable subset $L$ of $X$, and $h \colon K \to L$ a homeomorphism such that $h \Gamma_K h^{-1}=\Lambda_L$.

\begin{defn}
Assume that $\Delta \le \Gamma$ and $(\mcA,\Delta)$ is a $K$-compatible finite unit system, $\Sigma \le \Lambda$ and $(\mcB,\Sigma)$ is a $L$-compatible finite unit system.  

A Boolean algebra isomorphism $\Phi \colon \mcA \to \mcB$ is said to be $h$-\emph{compatible} if:
\begin{itemize}
\item $\Phi(A) \sim A$ for every $A \in \mcA$.
\item $\Phi \Delta_{\mcA} \Phi^{-1}= \Sigma_\mcB$ (we then say that $\Phi$ conjugates $(\mcA,\Delta)$ on $(\mcB,\Sigma)$).
\item For every $A \in \mcA$ we have $\Phi(A) \cap L= h(A \cap K)$.
\end{itemize}
\end{defn}

In the second bullet point above, recall that $\Delta_\mcA$ is the subgroup of $\mathrm{Aut}(\mcA)$  induced by the action of $\Delta$ (and similarly for $\Sigma_\mcB$).

The proof of Theorem \ref{t:ultimate_Krieger} goes through a back-and-forth argument. To make this argument work, it is enough to establish the following lemma (the proof of which is essentially the same as in \cite{Melleray2023}*{Lemma ~3.17} once one has Lemma \ref{l:compatible_clopen} in hand, though we repeat it here for the sake of completeness).

\begin{lemma}\label{l:Krieger_construction}
Assume that $(\mcA,\Delta)$, $(\mcB,\Sigma)$ are respectively $K$- and $L$-compatible finite unit systems with $\Delta \le \Gamma$, $\Sigma \le \Lambda$, and $\Phi \colon \mcA \to \mcB$ is a $h$-compatible Boolean algebra isomorphism. 

Let $(\mcA',\Delta')$ be a $K$-compatible finite unit system refining $(\mcA,\Delta)$ with $\Delta' \le \Gamma$. 

Then one can find a $L$-compatible finite unit system $(\mcB',\Sigma')$ refining $(\mcB,\Sigma)$, with $\Sigma' \le \Lambda$ and a $h$-compatible isomorphism $\Phi' \colon \mcA' \to \mcB'$ which extends $\Phi$.
\end{lemma}

\begin{proof}
For every orbit $\rho$ of the action of $\Delta$ on the atoms of $\mcA$, we choose a representative $A_\rho$. If $\rho$ intersects $K$, we choose $A_\rho$ so that $A_\rho \cap K \ne \emptyset$. 

For every $A \in \rho$, we denote by $\delta(\rho,A)$ the element of $\Delta$ which maps $A$ to $A_\rho$, $A_\rho$ to $A$, and is the identity everywhere else. This is an involution (and it is uniquely defined by definition of a unit system); in the particular case where $A=A_\rho$ we have $\delta(\rho,A_\rho)=\mathrm{id}$. Similarly, we denote $\sigma(\rho,A)$ the involution of $\Sigma$ exchanging $\Phi(A)$ and $\Phi(A_\rho)$ and which is the identity everywhere else.

For every $\rho$ we have 
\[A_\rho= \bigsqcup_{C \in \text{ atoms}(\mcA') : C \subseteq A_\rho} C\]

Let $C_1,\ldots,C_q$ denote the atoms of $\mcA'$ contained in $A_\rho$. If $q=1$ we let $U(C)=\Phi(C)$.

Assume that $q \ge 2$.
Applying Lemma \ref{l:compatible_clopen}, we find a clopen $U(C_1) \sim C_1$ contained in $\Phi(A_\rho)$ and such that $U(C_1) \cap L = h(C_1 \cap K)$; then a clopen $U(C_2) \sim C_2$ contained in $\Phi(A_\rho)$, disjoint from $U(C_1)$ and such that $U(C_2) \cap L = h(C_2 \cap K)$; and so on until $q-1$.

We have no choice but to set $U(C_q)= \Phi(A_\rho) \setminus \bigsqcup_{i=1}^{q-1} U(C_i)$. Since $h$ is bijective we have $U(C_q) \cap L= h(C_q \cap K)$; and by Lemma \ref{l:equidecomposability_on_clopens_is_induced_by_full_group} we also have $U(C_q) \sim C_q$.

We now have
\[\Phi(A_\rho) = \bigsqcup_{C \in \text{ atoms}(\mcA') : C \subseteq A_\rho} U(C)\]
where $U(C) \sim C$, and $U(C) \cap L = h(C \cap K)$ for all $C$. 

We define the algebra $\mcB'$ by setting as its atoms all $U(C)$, for $C$ an atom of $\mcA'$ contained in some $A_\rho$, as well as all $\sigma(\rho,A)(U(C))$ for $A \in \rho$ and $C$ contained in $A_\rho$. We obtain an isomorphism $\Phi' \colon \mcA' \to \mcB'$ by setting $\Phi'(C)=U(C)$ for every atom $C$ of $\mcA'$ contained in some $A_\rho$; and then
for any atom $C$ of $\mcA'$ contained in some $A \in \mcA$ whose $\Delta$-orbit is $\rho$, 
\[\Phi'(C)= \sigma(\rho,A) (U(\delta(\rho,A)(C))) \]

For every atom $C$ of $\mcA'$ we have $\Phi'(C) \cap L= h(C \cap K)$ by choice of $U(C)$, $K$-compatibility of $(\mcA,\Delta)$ and $L$-compatibility of $(\mcB,\Sigma)$.

We now need to construct the group $\Sigma'$. In the remainder of the proof, the letter $\tau$ always stands for an orbit of the action of $\Delta$ on the atoms of $\mcA'$, and the letter $\pi$ for an orbit of the action of $\Delta'$ on the atoms of $\mcA'$. For any $\tau$ there exists a unique $\pi$ which contains $\tau$. 

For any $\tau$ we choose a representative $B_\tau$, which intersects $K$ if some element of $\tau$ intersects $K$. Among all $B_\tau$ contained in a given $\pi$ we choose one $B_{\pi}$, and ask again that $B_{\pi}$ intersects $K$ if some element of $\pi$ intersects $K$. For every $\tau$ contained in $\pi$, we choose an involution $\lambda(\tau,\pi)\in \Lambda$ mapping $\Phi'(B_{\pi})$ to $\Phi'(B_\tau)$, and equal to the identity elsewhere. We also require that $\lambda(\tau,\pi)(\Phi'(B_{\pi}) \cap L)=\Phi'(B_\tau) \cap L$ if $B_\tau \cap K \ne \emptyset$ (which is possible thanks to Lemma \ref{l:compatible_clopen}).

Let $\Sigma'$ be the group generated by $\Sigma$ and $\{\lambda(\tau,\pi) : \tau \subset \pi\}$. Then $(\mcB',\Sigma')$ is a finite unit system (because we have added at most one link between any two $\Sigma$-orbits) and $\Phi'$ conjugates $(\mcA',\Delta')$ on $(\mcB',\Sigma')$.

We still need to show that $(\mcB',\Sigma')$ is $L$-compatible; so let $U$, $V$ be two atoms in $\mcB'$ belonging to the same $\Sigma'$-orbit and such that $U \cap L$, $V \cap L$ are both nonempty. 

There exists a $\Delta'$-orbit $\pi$, two $\Delta$-orbits $\tau_1$, $\tau_2$ contained in $\pi$ and involutions $\sigma_1, \sigma_2 \in \Sigma$ such that :
\[U= \sigma_1 \Phi'(B_{\tau_1})= \sigma_1 \lambda(\tau_1,\pi) \Phi'(B_{\pi}) \text{ and } V= \sigma_2 \Phi'(B_{\tau_2}) = \sigma_2\lambda(\tau_2,\pi) \Phi'(B_{\pi})\]
Any element $\sigma$ of $\Sigma'$ mapping $U$ to $V$ must coincide on $U$ with $\sigma_2 \lambda(\tau_2,\pi) \lambda(\tau_1,\pi) \sigma_1$. 

Using that $(\mcB,\Sigma)$ is $L$-compatible we then have:
\begin{align*}
\sigma(U \cap L) & =  \sigma_2 \lambda(\tau_2,\pi) \lambda(\tau_1,\pi)\sigma_1(U \cap L)\\
	   			  &= \sigma_2 \lambda(\tau_2,\pi) \lambda(\tau_1,\pi)\sigma_1 ((\sigma_1 \lambda(\tau_1,\pi) \Phi'(B_\pi)) \cap L) \\
				 & = \sigma_2 \lambda(\tau_2,\pi) \lambda(\tau_1,\pi)\sigma_1 (\sigma_1 \lambda(\tau_1,\pi) (\Phi'(B_\pi) \cap L)) \\
				 & = \sigma_2 \lambda(\tau_2,\pi) (\Phi'(B_\pi) \cap L ))\\
				 & = (\sigma_2 \lambda(\tau_2,\pi) \Phi'(B_\pi) ) \cap L 
				  =  V \cap L . \qedhere
\end{align*}
\end{proof}

We then obtain Theorem \ref{t:ultimate_Krieger} by using the same back-and-forth argument as in \cite{Krieger1979}*{Theorem~3.5}, which is also detailed in the proof of \cite{Melleray2023}*{Theorem~3.11}.

\section{A first absorption theorem}
\begin{defn}
Let $K$ be a closed subset of $X$, and $\Delta$ be a subgroup of $\Homeo(K)$. We denote by $R_\Gamma(K,\Delta)$ the finest equivalence relation $S$ which is coarser than $R_\Gamma$ and is such that $(x,\delta x) \in S$ for all $x \in K$ and all $\delta \in \Delta$.
\end{defn}

Informally, $R_\Gamma(K,\Delta)$ is obtained from $R_\Gamma$ by joining the $\Gamma$-orbits of $\Delta$-equivalent elements of $K$, and leaving untouched the $\Gamma$-orbits which do not intersect $K$.

\begin{lemma}\label{l:any_copy_will_do}
Let $K$, $K'$ be two $\Gamma$-malleable subsets, and $\Delta$, $\Delta'$ two ample groups over $K$.
Assume that there exists a homeomorphism $h \colon K \to K'$ such that $h\Gamma_K h^{-1}=\Gamma_{K'}$ and $(h \times h) R_\Delta=R_{\Delta'}$.

Then $R_\Gamma(K,\Delta)$ and $R_\Gamma(K',\Delta')$ are orbit equivalent.
\end{lemma}

\begin{proof}
Since $K$ and $K'$ are $\Gamma$-malleable, Theorem \ref{t:ultimate_Krieger} allows us to extend $h$ to an homeomorphism of $X$ such that $h \Gamma h^{-1} = \Gamma$. 
Then $(h \times h)R_\Gamma(K,\Delta)= R_\Gamma(K',\Delta')$.
\end{proof}

We now prove a particular case of the absorption theorem, which is already sufficient to reduce the theory of orbit equivalence of minimal $\Z$-actions to the corresponding theory for minimal ample groups. This theorem is a direct consequence of  our main result (Theorem \ref{t:absorption}). We could directly give the proof of the stronger result but it still seems worth giving a proof in this particular case, since it is technically much easier and some of the ideas of the more involved argument are already present. 

We fix a minimal homeomorphism $\varphi$ of $X$ and denote by $R_\varphi$ the associated equivalence relation on $X$. Given a closed subset $Y$ of $X$, we denote
\[F_Y(\varphi) = \bigcap_{y \in Y} F_y(\varphi) .\]
In other words, $F_Y(\varphi)$ consists of all elements of the topological full group $F(\varphi)$ which map the positive half-orbit of every element of $Y$ onto itself.

Following \cite{Melleray2023} we say that a closed subset $Y$ is \emph{sparse} if any two distinct points of $Y$ belong to distinct $\varphi$-orbits. This immediately implies that $Y$ is $F_Y$-malleable, and it is not hard to check that if $Y$ is sparse then the equivalence relation induced by the action of $F_Y(\varphi)$ is obtained from $R_\varphi$ by splitting the $\varphi$-orbit of every $y \in Y$ in its positive and negative parts (see Section 6 of \cite{Melleray2023} for proofs). 

Choose $y_0 \in X$ and include it in some sequence $(y_n)$ of elements of $X$ which converges to some $y_\infty \in X$ and is such that $Y=\{y_n : n \in \N\} \cup\{y_\infty\}$ is sparse. Now consider $Z = \varphi^{-1}(Y) \sqcup Y$. Consider also $Y' = Y \setminus \{y_0\}$, define $Z' = \varphi^{-1}(Y') \sqcup Y'$ and denote $g \colon Z \to Z'$ the homeomorphism which maps each $y_n$ to $y_{n+1}$, each $\varphi^{-1}(y_n)$ to $\varphi^{-1}(y_{n+1})$ and fixes $y_\infty$ and $\varphi^{-1}(y_\infty)$.

The equivalence relation $R_\varphi$ is the finest equivalence relation coarser than $R_{F_Y(\varphi)}$ and such that $y$ is equivalent to $\varphi^{-1}(y)$ for every $y \in Y$; in other words, it is obtained from $R_{F_Y(\varphi)}$ by gluing together the $F_Y(\varphi)$-orbits of $y$ and $\varphi^{-1}(y)$ for every $y \in Y$. Denote by $S$ the equivalence relation induced by $F_{y_0}(\varphi)$, and observe that $S$ is the relation obtained from $R_{F_{Y'}(\varphi)}$ by gluing together the $F_{\tilde Y}(\varphi)$-orbits of $y$ and $\varphi^{-1}(y)$ for every $y \in Y'$. 

We may then apply Lemma \ref{l:any_copy_will_do} to $g \colon Z \to Z'$ with $\Gamma=F_Y(\varphi)$, $\Delta = \{1,\pi \}$ where $\pi^2=\mathrm{id}$ and $\pi(y)= \varphi^{-1}(y)$ for all $y \in Y$, $\Gamma'=F_{Y'}(\varphi)$ and $\Delta'= \{1,\pi' \}$ where $\pi'$ is the restriction of $\pi$ to $Z'$ (here, both $\Gamma_Z$ and $\Gamma_{Z'}$ are trivial). We conclude that $R_\varphi$ and $S$ are orbit equivalent. 

Let us recap what we have just obtained; this result is due to Giordano, Putnam and Skau, see Theorem 2.3 of the seminal paper \cite{Giordano1995}.

\begin{theorem}
Let $\varphi$ be a minimal homeomorphism of $X$, and $y_0 \in X$. Then the equivalence relation $R_\varphi$ induced by $\varphi$ is orbit equivalent to the relation induced by $F_{y_0}(\varphi)$, that is, the equivalence relation where the orbit of $y_0$ has been split into its positive and negative semiorbits and all the other $R_\varphi$-classes have been unchanged.
\end{theorem} 

In particular, $R_\varphi$ is induced by a continuous action of an ample group. 

\section{The absorption theorem}
The argument in the previous section could be pushed further (for instance, one can recover a proof of Theorem 6.4 of \cite{Melleray2023}) but we want to establish a stronger absorption theorem, and that seems to require some significant extra work. 

We again fix an ample group $\Gamma$ over $X$ which acts minimally. 
We need to recall some more vocabulary from \cite{Melleray2023}.

\begin{defn}
We say that a clopen partition $\mcA=(A_{i,j})_{(i,j) \in I}$ is a \emph{$\sim_\Gamma$-partition} if:
\[\forall i,j,k \quad ((i,j) \in I \text{ and } (i,k) \in I ) \Rightarrow (A_{i,j} \sim_\Gamma A_{i,k})\]
Denoting by $I_i=\{j : (i,j) \in I\}$ we say that $\{A_{i,j} : j  \in I_i\}$ is an \emph{$\mcA$-orbit}.

A \emph{fragment} of an $\mcA$-orbit $(A_{i,j})_{j  \in I_i}$ is a family $(B_{i,j})_{j \in I_i}$ of nonempty clopen subsets of $X$ such that $B_{i,j} \subseteq A_{i,j}$ for all $j \in I_i$ and $B_{i,k} \sim_\Gamma B_{i,j}$ for all $j,k \in I_i$.
\end{defn}

For the reader who is already familiar with Kakutani--Rokhlin partitions and cutting-and-stacking arguments, orbits (and their fragments) are unordered counterparts of Kakutani--Rokhlin partitions. A fragment of an orbit is what one obtains by applying a cutting procedure to build a finer partition than what one started with.
Note that whenever $(\mcA,\Delta)$ is a finite unit system with $\Delta \le \Gamma$, one can view $\mcA$ as a $\sim_\Gamma$-partition by grouping together atoms of $\mcA$ which belong to the same $\Delta$-orbit.

We say that a $\sim_\Gamma$-partition $\mcB$ \emph{refines} another $\sim_\Gamma$-partition $\mcA$ if every $\mcB$-orbit can be written as a disjoint union of fragments of $\mcA$-orbits (intuitively, $\mcB$ has been obtained from $\mcA$ by cutting some $\mcA$-orbits, then grouping some fragments together). We note that we sometimes identify a clopen partition and the Boolean algebra it generates, which should cause no risk of confusion.

The following lemma is the key step of our proof of the absorption theorem (this is what allows the ``Hilbert--Bratteli hotel'' argument to go through). The idea is the following : in order to prove that $R_\Gamma$ can absorb some small extension of itself, we want to show that $R_\Gamma$ is itself obtained by repeating $\omega$ times the same small extension; first and foremost, one needs to show that $R_\Gamma$ can by reaslized as a prescribed small extension of $R_\Lambda$ for some minimal ample group $\Lambda$. 

\begin{lemma}\label{l:everyone_is_a_small_extension}
Let $Y$ be a compact, $0$-dimensional metric space and $\Delta \le \Sigma$ two ample groups over $Y$. 

There exists a minimal ample group $\Lambda \le \Gamma$, a closed subset $K$ which is both $\Lambda$- and $\Gamma$-malleable, and a homeomorphism $h \colon Y \to K$ such that $h \Delta h^{-1} = \Lambda_K$, $h \Sigma h^{-1} = \Gamma_K $ and $R_\Gamma=R_\Lambda(K,\Gamma_K)$.
\end{lemma}

\begin{proof}
We fix an exhaustive sequence $(\mcA_n,\Sigma_n)$ of finite unit systems for $(Y,\Sigma)$ and let $\Delta_n=\Sigma_n \cap \Delta$.

Denote by $h_n$ the number of atoms of $\mcA_n$. We pick an exhaustive sequence $(\mcB_n,\Gamma_n)$ of finite unit systems for $(X,\Gamma)$; for all $n$ we view $\mcB_n$ as a $\sim_\Gamma$-partition so that the $\mcB_n$- and $\Gamma_n$-orbits of every atom of $\mcB_n$ coincide. 

We also assume that for any $n$ there are more than $h_n$ distinct $\mcB_n$-orbits (which may be achieved by cutting them if necessary). Denoting by $k_n$ the number of atoms of $\mcB_n$, we may also ensure (because $\Gamma$ acts minimally) that every $\mcB_{n+1}$-orbit contains more than $(k_n+1)h_{n+1}$ disjoint fragments of every $\mcB_n$-orbit and that $k_{n} \ge n h_{n+1}$.

To initialize the construction, we add the assumption that $\mcA_0$, $\mcB_0$, $\Sigma_0$ and $\Gamma_0$ are all trivial.

\medskip
\textbf{Step $1$. Building a copy of $Y$.}
The first step consists of choosing, within each $\mcB_n$, some atoms which form a finer and finer approximation $K$ of $Y$; and doing so in such a way that these atoms generate a sequence of clopen partitions of $K$ which is induced by a copy of $\Sigma$. We also want to introduce ``cuts'' into the $\mcB_n$-orbits so as to mimic the way $\Sigma$-orbits of $\mcA_n$ are cut when one moves down from $\Sigma$ to $\Delta$.

By first choosing some elements of $\mcB_n$, then grouping them together with other elements of $\mcB_n$, we build a sequence of maps $\Phi_n \colon \mcA_n \to \mcB_n$ and a sequence of equivalence relations $\sim_n$ on $\mcB_n$; intuitively, $\Phi_n(\mcA_n)$ is our approximation (at step $n$) of the domain of the copy of $(Y,\Delta,\Sigma)$ that we are trying to build, and the restriction of $\sim_n$ to $\Phi_n(\mcA_n)$ mimics $\sim_{\Delta_n}$. 

Explicitly, our construction proceeds by enforcing the following conditions:

\begin{enumerate}
\item $\Phi_n \colon \mcA_n \to \mcB_n$ is injective, and for any $A,B \in \mcA_n$ we have 
\[\left(\Phi_n(A) \sim_{\Gamma_n} \Phi_n(B)\right) \Leftrightarrow (A \sim_{\Sigma_n} B) \, ; \ \left(\Phi_n(A) \sim_n \Phi_n(B)\right) \Leftrightarrow (A \sim_{\Delta_n} B) \]

\item For any $U,V \in \mcB_n$, $U \sim_n V \Rightarrow U \sim_{\Gamma_n} V$.

\item  If $U \in \mcB_{n}$ is such that the $\mcB_n$-orbit $\tau_U$ of $U$ does not intersect the image of $\Phi_n$, then all elements of $\tau_U$ are $\sim_{n}$-equivalent. Else, every element of $\tau_U$ is $\sim_{n}$-equivalent to some $\Phi_{n}(A)$ for $A \in \mcA_{n}$. 
\item If $A \in \mcA_n$, $B \in \mcA_{n+1}$ are such that $B \subseteq A$ then $\Phi_{n+1}(B) \subseteq \Phi_n(A)$.

Denote by $\mcC_n$ the $\sim_\Gamma$-partition whose atoms are the same as those of $\mcB_n$ and whose orbits are the $\sim_n$-classes. Each $\mcB_n$-orbit is a union of $\mcC_n$-orbits.

\item Given $A \in \mcA_{n+1}$, the $\sim_{n+1}$-class of $\Phi_{n+1}(A)$ is obtained by joining one distinguished fragment of the $\mcC_{n}$-orbit of every $\Phi_{n+1}(A')$ for $A' \in \Delta_{n+1}A$ and fragments of the other $\mcB_n$-orbits, so that a fragment of every $\mcB_n$-orbit appears at least once. In particular, $\mcC_{n+1}$ refines $\mcC_n$ and each $\mcC_{n+1}$-orbit contains a fragment of every $\mcC_n$-orbit.
\end{enumerate}

To see that this is indeed possible, assume that this construction has been carried out up to some $n$ (for $n=0$ there is nothing to do). Below, for $A$ an atom of $\mcA_{n+1}$ we denote by $A'$ the unique atom of $\mcA_n$ which contains $A$.

We first define an injective $\Phi_{n+1} \colon \mcA_{n+1} \to \mcB_{n+1}$. 
For any $A \in \mcA_{n+1}$ we ask that $\Phi_{n+1}(A) \subseteq \Phi_n(A')$, and require also that, for any $A,B \in \mcA_{n+1}$, $\Phi_{n+1}(A)$ and $\Phi_{n+1}(B)$ belong to the same $\mcB_{n+1}$-orbit iff $A$ and $B$ belong to the same $\Sigma_{n+1}$-orbit. This is feasible because the number of $\mcB_{n+1}$-orbits is greater than $h_{n+1}$ (hence larger than the number of $\Sigma_{n+1}$-orbits of atoms of $\mcA_{n+1}$) and, for any $A' \in \mcA_n$, in any $\mcB_{n+1}$-orbit there are more than $h_{n+1}$ atoms contained in $A'$, so that we can guarantee the injectivity of $\Phi_{n+1}$.

Next, fix a $\mcB_{n+1}$-orbit $\tau$ which intersects $\Phi_{n+1}(\mcA_{n+1})$. We
start building $\sim_{n+1}$ on $\tau$, by choosing for each $A \in \mcA_{n+1}$ such that $\Phi_{n+1}(A) \in \tau$ a fragment of the $\mcC_n$-orbit of $\Phi_n(A')$ which contains $\Phi_{n+1}(A)$ and is contained in $\tau$. We do this in such a way that 
$\Phi_{n+1}(A)$ and $\Phi_{n+1}(B)$ belong to the same fragment iff $A' \in \Delta_{n} B'$. 
Again, this is feasible because $\tau$ contains many fragments of every $\mcB_n$-orbit, hence of every $\mcC_n$-orbit. Then we include in the $\sim_{n+1}$-class of $\Phi_{n+1}(A)$ the union of the fragments we have chosen for all $B \in \Delta_{n+1}A$ (and nothing else at this point).

The union of fragments of $\mcC_n$-orbits on which we have so far defined $\sim_{n+1}$ in $\tau$ is actually a union of fragments of $\mcB_n$-orbits: indeed, for each $\mcB_n$-orbit $\rho$ which intersects $\Phi_n(\mcA_n)$ we have selected whole fragments, because $\rho$ is the union of $\sim_n$-classes of all $\Phi_n(A')$ contained in $\rho$ (and the fragments of other $\mcB_n$-orbits have not yet been taken under consideration). Thus, at this point, each partial $\sim_{n+1}$-equivalence class in $\tau$ consists of a union of a most $h_{n+1}$ fragments of $\mcB_{n}$-orbits. Inside $\tau$, we have at least $(k_n+1)h_{n+1}$ such fragments of each $\mcB_n$-orbit at our disposal, hence we can distribute the remaining fragments among the (at most $h_{n+1}$) disjoint partial $\sim_{n+1}$-classes contained in $\tau$ in such a way that each $\sim_{n+1}$-class contains at least one fragment of every $\mcB_n$-orbit, and every atom of $\tau$ is in the $\sim_{n+1}$-class of some $\Phi_{n+1}(A)$.

Performing the above procedure for every $\tau$ which intersects $\Phi_{n+1}(\mcA_{n+1})$, we obtain the desired relation $\sim_{n+1}$.

For all $n \in \N$ we let $K_n= \bigsqcup_{A \in \mcA_n} \Phi_n(A)$. This is a clopen set, $K_{n+1} \subseteq K_n$ for all $n$ and we define $K=\bigcap_{n \in \N} K_n$. 

The sequence $(\Phi_n)_n$ induces an isomorphism $\Phi \colon \Clopen(Y) \to \Clopen(K)$, and we let $h \colon Y \to K$ be the corresponding homeomorphism. 

\medskip
\textbf{Step $2$. Defining $\Lambda$.} We now want to use our previous work in order to define a copy $\Lambda$ of $\Delta$.

Denote by $\Lambda_n$ the group of permutations of the atoms of $\mcB_n$ which map each $\sim_n$-class to itself. We define embeddings $i_n \colon \Lambda_n \to \Lambda_{n+1}$ so that (among other properties) for any $\lambda \in \Lambda_n$ and any $B \in \mcB_n$, if $B_1,\ldots,B_p$ are the atoms of $\mcB_{n+1}$ such that $B= \bigsqcup_{i=1}^p B_i$ then $\lambda (B)= \bigsqcup_{i=1}^p i_n(\lambda)(B_i)$.

To explain how $i_n$ is defined, we fix a $\sim_n$-orbit $\tau$ and assume first that there is $A \in \mcA_n$ such that $\Phi_n(A) \in \tau$. Given another atom $U$ of $\tau$, denote $\lambda_{U}$ the unique involution in $\Lambda_n$ which maps $\Phi_n(A)$ to $U$ and leaves all other elements of $\mcB_n$ fixed. 

First, we consider the case where $U=\Phi_n(\delta A)$ with $\delta \in \Delta_n$. We can write $A=\bigsqcup_{i=1}^p A_i$ with $A_i \in \mcA_{n+1}$ and set 
\[ i_n(\lambda_U)(\Phi_{n+1}(A_i))=\Phi_{n+1}(\delta A_i) \]
We still have to define $i_n(\lambda_U)$ on some atoms of $\mcB_{n+1}$ contained in $\Phi_n(A),\Phi_n(\delta A)$. In each $\mcC_{n+1}$-orbit there remain as many atoms contained in $\Phi_n(\delta A)$ where $i_n(\lambda_U)$ has not been defined as there are such atoms contained in $\Phi_n(A)$ (because $\mcC_{n+1}$ refines $\mcC_n$, every $\mcC_{n+1}$-orbit has as many atoms contained in $\Phi_n(A)$ and $\Phi_n(\delta A)$). We may thus match those atoms arbitrarily to define $i_n(\lambda_U)$ so that it maps each of these atoms to an atom in the same $\sim_{n+1}$-class.

The second case is when $U$ is not in the image of $\Phi_n(\mcA_n)$. Again, each $\mcC_{n+1}$-orbit has as many atoms contained in $U$ and in $\Phi_n(A)$, and we match those arbitrarily to define $i_n(\lambda_U)$.

If our $\sim_n$-orbit $\tau$ does not intersect the image of $\Phi_n$, then we choose $V \in \tau$ at random, define $\lambda_U$ as above (for $U$ another element of $\tau$) and define $i_n(\lambda_U)$ as in the previous paragraph (any involution which extends $\lambda_U$ and fixes setwise each $\sim_{n+1}$-class will do the job).

Once all this is done, we have completely defined $i_n \colon \Lambda_n \to \Lambda_{n+1}$. 

The inductive limit $\Lambda$ of the sequence $(\Lambda_n,i_n)_n$ naturally acts on $\Clopen(X)$, and we view it as a subgroup of $\Homeo(X)$. By construction, this group is ample. 

Denoting also by $i_n$ the embedding of $\Lambda_n$ in $\Lambda$, the $i_n(\Lambda_n)$-orbit of every element of $\mcC_{n}$ coincides with its $\mcC_{n}$-orbit. Given a nonempty $U \in \Clopen(X)$, let $n$ be such that $U $ is a union of atoms of $\mcC_n$. Since every  $\mcC_{n+1}$-orbit contains a fragment of every $\mcC_n$-orbit, hence an atom contained in $U$, we have $i_{n+1}(\Lambda_{n+1})U=X$. We conclude that $\Lambda$ acts minimally on $X$. 

In each $\mcC_{n+1}$-orbit there are at least $k_n$ elements, and at most $h_{n+1}$ of those belong to $K_{n+1}$; since $h_{n+1}/k_n \xrightarrow[n \to + \infty]{} 0$, we conclude that $\mu(K)=0$ for every $\mu \in M(\Lambda)$, so $K$ is $\Lambda$-thin. 

By construction, if $A,B \in \mcA_n$ and $\lambda \in \Lambda_n$ are such that $\lambda \Phi_n(A)=\Phi_n(B)$
we can write $A= \bigsqcup_{i=1}^p A_i$, $B=\bigsqcup_{i=1}^p B_i$ with $A_i, B_i \in \mcA_{n+1}$ and
$i_n(\lambda)(\Phi_{n+1}(A_i))=\Phi_{n+1}(B_i)$. By induction, it follows that $i_n(\lambda)(\Phi_n(A) \cap K)= \Phi_n(B) \cap K$.
This implies that $K$ is $\Lambda$-étale. Finally, $K$ is $\Lambda$-malleable.

Our definition of $i_n$ also ensures that $h \Delta h^{-1}=\Lambda_{K}$ (where $h \colon Y \to K$ is the homeomorphism defined at the end of the first step). 

Unfortunately we are not quite done yet: at this stage there is no reason why $\Lambda$ should be contained in $\Gamma$. To remedy this, we replace $\Gamma$ by a conjugate which we construct now.

\medskip
\textbf{Step $3$. Defining a conjugate of $\Gamma$.}
We now want to build an ample group $\tilde \Gamma$ which contains $\Lambda$ and is such that one can write $\tilde \Gamma=\bigcup_n \tilde \Gamma_n$, so that the $\tilde \Gamma_n$-orbit of every element of $\mcB_n$ coincides with its $\mcB_n$-orbit. We also need to guarantee that $h \Sigma h^{-1} = \tilde{\Gamma}_K$. Again this group is obtained as an inductive limit of finite permutation groups.

For all $n$, we let $\tilde \Gamma_n$ denote the group of all permutations of $\mcB_n$ which map each $\mcB_n$-orbit to itself; note that $\Lambda_n$ is a subgroup of $\tilde \Gamma_n$. Similarly to what we did above, we build embeddings $j_n \colon \tilde \Gamma_n \to \tilde \Gamma_{n+1}$ to form an adequate inductive limit. 

First, we ask that $j_n$ coincide with $i_n$ on $\Lambda_n$. Next, let $\tau$ be a $\mcB_n$-orbit; if $\tau$ is also a $\mcC_n$-orbit then we already know how to extend elements of $\tilde \Gamma_n$ whose support is contained in $\tau$ (those belong to $\Lambda_n$ so we simply apply $i_n$). So assume that $\tau$ intersects the image of $\Phi_n$ and let $A_1,\ldots,A_p \in \mcA_n$ be such that $p \ge 2$ and $\tau$ is the disjoint union of the $\mcC_n$-orbits of $\Phi_n(A_1),\ldots,\Phi_n(A_p)$.

The $\Sigma_n$-orbit of $A_1$ is the disjoint union of the $\Delta_n$-orbits of $A_1, \ldots,A_p$. For $i \in \{2,\ldots,p\}$ we let $\sigma_i$ be the involution in $\Sigma_n$ which maps $A_1$ to $A_i$ and fixes all other atoms, and denote $\gamma_i$ the element of $\tilde \Gamma_n$ which maps $\Phi_n(A_1)$ to $\Phi_n(A_i)$ and fixes all other atoms. 

Fix $i \in \{2,\ldots,p\}$. Write $A_1=\bigsqcup_{k=1}^q A_{1,k}$, $A_i=\bigsqcup_{k=1}^q \sigma_i(A_{1,k})$ with $A_{1,k} \in \mcA_{n+1}$, then set 
\[ j_n(\gamma_i)(\Phi_{n+1}(A_{1,k}))=\Phi_{n+1}(\sigma_i A_{1,k}) \]

On each atom of $\mcB_{n+1}$ which does not intersect $\Phi_n(A_1) \sqcup \Phi_n(A_i)$ we must have $j_n(\gamma_i)$ coincide with the identity. Now, we observe that in each $\mcC_{n+1}$-orbit there are as many atoms contained in $\Phi_{n}(A_1)$ on which $j_n(\gamma_i)$ has not yet been defined as there are such atoms in $\Phi_n(A_i)$. 
This follows from the way $\sim_{n+1}$ has been defined above: so far we have defined what $j_n(\gamma_i)$ does to atoms contained in the $\mcC_n$-fragments intersecting $\Phi_{n+1}(\mcA_{n+1})$, and the remainder consists of unions of fragments of $\mcB_n$-orbits.
We then match those atoms arbitrarily to define $j_n(\gamma_i)$ so that, on each atom $U$ which is not in $\Phi_{n+1}(\mcA_{n+1})$, $j_n(\gamma_i)U$ is $\sim_{n+1}$-equivalent to $U$ (this property is used at the end of the proof).

We apply this procedure to every $\mcB_n$-orbit on which $\sim_n$ is nontrivial. 
This finally defines $j_n$, and we obtain an ample group $\tilde \Gamma$ by considering the inductive limit of $(\tilde \Gamma_n,j_n)$. By construction we have $h \Sigma h^{-1}= \tilde{\Gamma}_K$. Denoting $j_n$ the embedding of $\tilde \Gamma_n$ into $\tilde \Gamma$, the $j_n(\tilde \Gamma_n)$-orbit of every element of $\mcB_n$ coincides with its $\Gamma_n$-orbit, and $(\mcB_n,j_n(\tilde \Gamma_n))_n$ is an exhaustive sequence of unit systems for $\tilde \Gamma$. It follows from Krieger's theorem that $\Gamma$ and $\tilde \Gamma$ are conjugate.

\medskip
\textbf{Step $4$. Checking that the construction satisfies the desired conditions.}
By definition, $\tilde \Gamma$ contains $\Lambda$. Since we already noted that $\mu(K)=0$ for all $\mu \in M(\Lambda)$, we also have $\mu(K)=0$ for all $\mu \in M(\tilde \Gamma)$.

Similarly to Step $2$, the definition of each $j_n$ also ensures that for $A,B \in \mcA_n$ and $\gamma \in \tilde \Gamma_n$ such that $\gamma \Phi_n(A)=\Phi_n(B)$ we have $j_n(\gamma)(\Phi_n(A) \cap K)=\Phi_n(B) \cap K$. Hence $K$ is $\tilde \Gamma$-malleable. We  have already noted that $h \Delta h^{-1} = \Lambda_K$ and $h \Sigma h^{-1} = \tilde \Gamma_K$.

Since $\Lambda \le \tilde \Gamma$ we have that $R_{\Lambda}(K,\tilde{\Gamma}_K)$ is contained in $R_{\tilde \Gamma}$. 
To see the converse inclusion, let $\gamma \in \tilde \Gamma$ and pick $n$ such that $\gamma \in j_n(\tilde \Gamma_n)$. Let $U$ be an atom of $\mcB_n$. If $\gamma U$ and $U$ belong to the same $\mcC_n$-orbit then $\gamma$ coincides on $U$ with an element of $\Lambda$, so $\gamma x \in \Lambda x$ for every $x \in U$. If $\gamma U$ and $U$ belong to different $\mcC_n$-orbits, then there exist $\lambda_1, \lambda_2 \in i_n(\Lambda_n)$, $A_1,A_2 \in \Phi_n(\mcA_n)$ and $\sigma \in \Sigma_n \setminus \Delta_n$ such that $\sigma(A_1)=A_2$ and $\lambda_1 U = \Phi_n(A_1)$, $\lambda_2 \Phi_n(A_2)=\gamma U$. Let $\gamma'$ be the involution of $\tilde \Gamma_n$ which maps $\Phi_n(A_1)$ to $\Phi_n(A_2)$ and fixes all other atoms. By construction, we have that $j_n(\gamma')(K)=K$, and for every $x \in X \setminus K$ we have $j_n(\gamma')(x) \in \Lambda x$. Since $\gamma$ coincides on $U$ with $i_n(\lambda_2)j_n(\gamma')i_n(\lambda_1)$, we see that $(x,\gamma x) \in R_\Lambda(K,\tilde{\Gamma}_K)$ for every $x \in U$. We finally conclude that $R_{\tilde \Gamma}=R_\Lambda(K,\tilde{\Gamma}_K)$.

We have obtained the desired result for $\tilde \Gamma$ instead of $\Gamma$; since $\Gamma$ and $\tilde \Gamma$ are conjugate in $\Homeo(X)$ this is enough.
\end{proof}

We are finally ready to prove the main result. We repeat that the idea is to use Lemma \ref{t:absorption} to see $R_\Gamma$ as having been obtained by absorbing countably many ``copies'' of  $(K,\Delta)$; intuitively, absorbing one more copy which is independent from the previous ones should not change the orbit equivalence class of $R_\Gamma$.

\begin{theorem}[The absorption theorem]\label{t:absorption} Let $\Gamma$ be a minimal ample group, $K$ a $\Gamma$-malleable subset and $\Sigma$ an ample group over $K$ which contains $\Gamma_K$. Then $R_\Gamma(K,\Sigma)$ is orbit equivalent to $R_\Gamma$.
\end{theorem}

\begin{proof}
Form a compact metric space $Y= (\bigsqcup_{n \in \N} K_n) \sqcup \{y\}$, where each $K_n$ is a clopen subset of $Y$ homeomorphic to $K$ via some homeomorphism $h_n \colon K \to K_n$ and the diameter of $K_n$ as well as $d(K_n,y)$ both converge to $0$ (in the previous section, when we wanted to absorb a point, $K_n$ was a singleton; we adapt that strategy and this part of the argument is very similar to what we did above).

For $n \in \N$ let $\Delta_n = h_n \Gamma_K h_n^{-1}$, which we view as a subgroup of $\Homeo(Y)$ (whose elements act trivially outside $K_n$); denote by $\Delta$ the full subgroup of $\Homeo(Y)$ generated by $\bigcup_{n \in \N} \Delta_n$. It is ample because every element of $\Delta$ fixes pointwise a neighborhood of $y$, and for any $n$ any $\delta \in \Delta$ coincides on $K_n$ with an element of $\Delta_n$.

Define $\Pi_0=\Delta_0=h_0 \Gamma_K h_0^{-1}$ and, for $n \ge 1$, $\Pi_n= h_n \Sigma h_n^{-1}$ (viewed as subgroups of $\Homeo(Y)$ as above) then let $\Pi$ be the full subgroup of $\Homeo(Y)$ generated by $\bigcup_{n \in \N} \Sigma_n$. Again, it is ample.

Then use Lemma \ref{l:everyone_is_a_small_extension} to find a minimal ample group $\Lambda$, a closed subset $Z$ which is both $\Lambda$- and $\Gamma$-malleable and a homeomorphism $g \colon Y \to Z$ such that $g\Delta g^{-1}=\Lambda_Z$, $g \Pi g^{-1}= \Gamma_Z$ and $R_\Gamma=R_\Lambda(Z,\Gamma_Z)$.

Write $Z=\bigcup_{n \in \N} Z_n \cup \{g(y)\}$ where for each $n$ $Z_n= g(K_n)=g \circ h_n(K)$. 

Denote $Z'=Z \setminus Z_0$. It is $\Lambda$-malleable since it is clopen in $Z$ and $\Lambda_Z$-invariant. Note that $\Gamma_{Z_0}=g \Pi_{0}g^{-1}=g \Delta_0 g^{-1} = \Lambda_{Z_0}$. So $R_\Lambda$ and $R_\Gamma$ coincide on $Z_0$ and it follows that $R_\Gamma=R_\Lambda(Z,\Gamma_Z)=R_\Lambda(Z',\Gamma_{Z'})$.

Similarly, $Z_0$ is $\Gamma$-malleable (again because $Z_0$ is clopen in $Z$ and $\Gamma_Z$-invariant). Define $\Theta_0 = (gh_0) \Sigma (gh_0) ^{-1}$. This is an ample group over $Z_0$.

Denote by $\Theta$ the ample subgroup of $\Homeo(Z)$ generated by $\Theta_0 \cup \bigcup_{n \ge 1} \Gamma_{Z_n}$. Let $f \colon Z \to Z'$ be such that for all $n$ and all $x \in K$ one has $f(gh_n(x))=gh_{n+1}(x)$. This is a homeomorphism.

Then $f \Lambda_Z f^{-1} = \Lambda_{Z'}$ and $f \Theta f^{-1}= \Gamma_{Z'}$. Since $Z'$ is $\Lambda$-malleable, we can apply Lemma \ref{l:any_copy_will_do} and obtain that $R_\Lambda(Z,\Theta)$ and $R_\Lambda(Z',\Gamma_{Z'})$ are orbit equivalent. In other words, $R_\Gamma(Z_0,\Theta_0)$ and $R_\Gamma$ are orbit equivalent.

We have that $Z_0$ is $\Gamma$-malleable, $(gh_0) \Gamma_K (gh_0)^{-1}=\Gamma_{Z_0}$ and 
$(gh_0) \Sigma (gh_0)^{-1}= \Theta_0$. Applying Lemma \ref{l:any_copy_will_do} once more, we get that $R_\Gamma(Z_0,\Theta_0)$ and $R_\Gamma(K,\Sigma)$ are orbit equivalent. Hence $R_\Gamma$ is orbit equivalent to $R_\Gamma(K,\Sigma)$, and the absorption theorem is proved.
\end{proof}

\section{Connection to Matui's absorption theorem}
We briefly sketch (without proofs) why Theorem \ref{t:absorption} enables one to recover Matui's absorption theorem. Here it is assumed that the reader already knows the theory of AF equivalence relations and Bratteli diagrams; we use liberally the terminology of Matui's article \cite{Matui2008}.

First, we recall that a minimal AF relation $R$ on the Cantor space $X$ is the same thing as an equivalence relation induced by a minimal action of an ample group $\Gamma$ over $X$, endowed with the topology associated to this action; this follows for instance from Theorem 3.9 of \cite{Giordano2004} (see also the sixth chapter of Putnam's book \cite{Putnam2018}). Given such an ample group $\Gamma$, saying that a closed subset $Y$ of $X$ is $\Gamma$-étale corresponds to the statement that the restriction of $R$ to $Y$, with the relative topology induced from the topology of $R$, is an étale equivalence relation on $Z$ (see Theorem 3.11 of \cite{Giordano2004}). 

Given a minimal AF relation $R$ on a compact, $0$-dimensional $Y$ and an open subrelation $R'$ of $R$, it follows from the argument in the proofs of Proposition 3.12(ii) and Theorem 3.9 of \cite{Giordano2004} that there exists an ample group $\Delta$ over $Y$ which induces $R'$, and given such a $\Delta$ one can produce an ample group $\Gamma$ over $Y$containing $\Delta$ which induces $R$.

Now, let $R$ be a minimal AF equivalence relation on $X$, $Y$ be a closed, $R$-étale and $R$-thin subset and assume that $Q \subset Y \times Y$ is an AF equivalence relation which is an étale extension of $R_{|Y}$. It follows from the previous discussion that there exists an ample subgroup $\Gamma$ over $X$ which induces $R$, and an ample subgroup $\Sigma$ over $Y$ which contains $\Gamma_Y$ and which induces $Q$. Then the proof of Theorem \ref{t:absorption} produces a homeomorphism $h$ of $X$ such that:
\begin{itemize}
\item $h \times h (R \vee Q)= R$.
\item $h(Y)$ is closed, $R$-étale and $R$-thin (equivalently, $\Gamma$-malleable).
\item $h_{|Y} \times h_{|Y}$ is a homeomorphism from $Q$ to $R_{|h(Y)}$ (because $h \Sigma h^{-1}=\Gamma_{h(Y)}$).
\end{itemize}

That is the statement of Matui's absorption theorem. Conversely, it is clear that Matui's absorption theorem implies Theorem \ref{t:absorption} as stated (though the proof of Theorem \ref{t:absorption} actually gives a somewhat stronger result if one tracks down what happens to the ample subgroups involved, and it is not immediately clear whether this can also be easily recovered from Matui's argument, although that seems quite likely).

\newpage

\begin{appendices}
        \begin{center}\section*{Appendix}
        
        \textbf{Corrigendum to the paper ``From invariant measures to orbit equivalence, via locally finite groups'', by J. Melleray and S. Robert}
\end{center}        
\setcounter{section}{1}
\setcounter{theorem}{0}

As explained in the introduction, the proof of the classification theorem of Giordano, Putnam and Skau given in \cite{Melleray2023} has a gap. The same issue occurs twice, so that the proofs of the absorption theorem \cite{Melleray2023}*{Theorem~5.1} as well as the proof of the classification theorem \cite{Melleray2023}*{Theorem~5.2} are incorrect. The absorption theorems that one can obtain from the arguments of \cite{Melleray2023} (one can for instance recover \cite{Melleray2023}*{Theorem~6.4}) do not appear to be sufficient to fix the proof of the classification theorem. 

Let us explain the issue briefly; here and below we reuse the notations and terminology of \cite{Melleray2023}. The idea of our proof of the classification theorem is to start from a minimal ample group $\Gamma$ and build a $\Gamma$-sparse $K$ (the set of ``singular'' points in that proof), an involution $\pi \colon K \to K$ and a minimal ample group $\Lambda$, generated as a full group by $\Gamma$ and $\pi$, so that $R_\Lambda$ is obtained from $R_\Gamma$ by gluing together the orbits of $x$ and $\pi(x)$ for every $x \in K$. Unfortunately, the proof given in \cite{Melleray2023} does not achieve this; the reason is that there are redundancies in the definition so that the full group $\Lambda$ generated by $\Gamma$ and $\pi$ as constructed in that proof need not be ample. To fix this and ensure that $\Lambda$ is ample one needs to avoid those redundancies; but then one cannot force the singular points to belong to different $\Gamma$-orbits. One can however guarantee that the set $K$ of singular points is $\Gamma$-malleable, and then using Matui's absorption theorem (i.e.~Theorem \ref{t:absorption}) allows the argument to go through.

We now explain how to combine Theorem \ref{t:absorption} with the argument of \cite{Melleray2023} to obtain a proof of the classification theorem for minimal ample groups. 

We fix a minimal ample group $\Gamma$, and assume that $\sim_\Gamma$ and $\sim_\Gamma^*$ do not coincide. We want to prove that $R_\Gamma$ is orbit equivalent to an action induced by a minimal ample group $\Lambda$ with the property that $\sim_\Lambda$ and $\sim_\Lambda^*$ both coincide with $R_\Gamma^*$ (we say that $\Lambda$ is \emph{saturated} when $\sim_\Lambda$ and $\sim_\Lambda^*$ coincide), so as to employ Krieger's theorem one last time to prove the classification theorem. To reach this objective, we want to prove that there exists such an ample group $\Lambda$ and a malleable subset $K$ such that $\Gamma_{|K} \subseteq \Lambda$ and $R_\Lambda=R_{\Gamma}(K,\Lambda_{|K})$, in order to apply the absorption theorem. We do this by building an involution $\pi$ and a conjugate $\tilde \Gamma$ of $\Gamma$ in such a way that $\tilde \Gamma$ and $\pi$ generate our desired ample group $\Lambda$, and $(x,\pi(x)) \in R_{\tilde \Gamma}$ for every $x$ outside of a malleable set. In terms of partitions, $K$ is created from atoms that one is ``forced'' to group together because of pairs of clopen subsets which are $\sim_{\Gamma}^*$-equivalent even though they are not $\sim_\Gamma$-equivalent (recall that we want to make such pairs $\Lambda$-equivalent).

We choose an enumeration $(U_n,V_n)_n$ of all pairs of $\sim_\Gamma^*$-equivalent clopen subsets of $X$ and assume that $U_0$, $V_0$ are disjoint, $U_0 \not \sim_\Gamma V_0$ and $U_0 \cup V_0 \neq X$.

First, we slightly modify \cite{Melleray2023}*{Lemma~5.3} to obtain the following:

\begin{lemma}\label{l:main_construction}
We may build a sequence of $\sim_\Gamma$-partitions $(\mcA_n)$, with distinguished orbit pairs $O(\alpha_1^n),\ldots,O(\alpha_{k_n}^n), O(\beta_1^n), \ldots O(\beta_{k_n}^n)$ $(1 \le k_n$ for all $n$) satisfying the following conditions. 

\begin{enumerate}
\item $k_0=1$, $\alpha_1^0=U_0$, $\beta_1^0=V_0$ and $\mcA_0 = \{\alpha_1^0,\beta_1^0,X \setminus (\alpha_1^0 \cup \beta_1^0)\}$ (three orbits of cardinality $1$).
\end{enumerate}
For all $n$ one has:
\begin{enumerate}
\setcounter{enumi}{1}
\item $\mcA_{n+1}$ refines $\mcA_n$.
\item If $U_n \sim_\Gamma V_n$ then $U_n$ and $V_n$ are $\mcA_n$-equivalent.
\item $(\alpha_1^n,\ldots,\alpha_{k_n}^n, U_{n+1})$ and $(\beta_1^n,\ldots,\beta_{k_n}^n,V_{n+1})$ are almost $\mcA_{n+1}$-equivalent, as witnessed by the exceptional orbits 
$$O(\alpha_1^{n+1}),\ldots,O(\alpha_{k_{n+1}}^{n+1}) \ , \ O(\beta_1^{n+1}),\ldots,O(\beta_{k_{n+1}}^{n+1}).$$
\item For all $i$ $\alpha_i^n \not \sim_\Gamma \beta_i^n$.
\item For all $i$, $\alpha_i^{n+1}$ is contained in $\alpha_1^n$ and  $\beta_i^{n+1}$ is contained in $\beta_1^n$.
\item Let $h_n$ be the number of atoms of $\mcA_n$; denote 
\begin{eqnarray*}
N^{(n)}_i&=& \max\{|n_O(\alpha^n_i) - n_O(\beta^n_i)| : O \text{ is a } \mcA_{n+1}-\text{orbit} \} \quad (i \le k_n), \\
N^{(n)} &=& \sum_{i=1}^{k_n} N^{(n)}_i .
\end{eqnarray*}
Then every exceptional $\mcA_{n+1}$-orbit contains more than $(n+1) h_n (N^{(n)}+2)$ fragments of every $\mcA_n$-orbit.
\end{enumerate}
\end{lemma}

The slight modification alluded to above is twofold (and of a purely technical nature). First, the penultimate point in the Lemma's statement is not present in \cite{Melleray2023} (but should be, given how the construction proceeds). It is not hard to enforce this condition: since every $\mcA_{n+1}$-orbit contains a fragment of every $\mcA_n$-orbit, we may choose  each $\alpha_i^{n+1}$ so that it is contained in any prescribed nonempty clopen $U$, and similarly for $\beta_i^{n+1}$. The numerical constant in the last point also changed, but this is purely cosmetic, see the end of the proof of \cite{Melleray2023}*{Proposition~4.10}.

As in \cite{Melleray2023} we define a sequence of $\sim_\Gamma^*$-partitions $\mcB_n$ by joining together the $\mcA_n$-orbits of $\alpha_i^n$ and $\beta_i^n$ for each $i \in \{1,\ldots,k_n\}$ and leaving the other orbits unchanged. Then $\mcB_{n+1}$ refines $\mcB_n$ for all $n$.

We now have in hand two sequence of partitions. Similarly to what we did to prove Lemma \ref{l:everyone_is_a_small_extension}, we build two ample groups associated to these partitions (one of them is a conjugate of $\Gamma$, and the other is a larger, saturated ample group $\Lambda$ such that $\sim_{\Lambda}=\sim_{\Gamma}^*$). That is, we replace \cite{Melleray2023}*{Lemma~5.4} with the following result.

\begin{lemma}
Denote by $\tilde \Gamma_n$ the group of permutations of $\mcA_n$ which map each $\mcA_n$-orbit to itself, and by $\Lambda_n$ the group of permutations of $\mcB_n$ which map each $\mcB_n$-orbit to itself. Then we can build embeddings $i_n \colon \tilde \Gamma_n \to \tilde \Gamma_{n+1}$ and $j_n \colon \Lambda_n \to \Lambda_{n+1}$ such that, for all $n$:
\begin{enumerate}
\item The action of $i_n(\tilde \Gamma_n)$ on $\mcA_{n+1}$ extends the action of $\tilde \Gamma_n$ on $\mcA_n$, and the action of $j_n(\Lambda_n)$ on $\mcB_{n+1}$ extends the action of $\Lambda_n$ on $\mcA_n$.
\item $j_n$ coincides with $i_n$ on $\tilde \Gamma_n$.
\item Let $\pi_0$ be the involution in $\Lambda_0$ mapping $\alpha_1^0$ to $\beta_1^0$ and define for $n \ge 1$ $\pi_n=j_{n-1} \circ \ldots \circ j_0(\pi_0)$. Then $\pi_n(\alpha_j^n)=\beta_j^n$ for all $j$.
\item Say that an atom $\alpha$ of $\mcA_n$ is \emph{singular} if $\pi_n(\alpha) \not \in \tilde \Gamma_n \alpha$ and let $K_n$ denote the union of all singular atoms of $\mcA_n$. Then $\mu(K_{n+1}) \le \frac{1}{n+1}$ for all $\mu \in M(\tilde \Gamma)$ and all $n$.
\end{enumerate}
\end{lemma}

\begin{proof}
To build the sequence of embeddings $i_n$, one can use an argument similar to the one we used in the second step of the proof of Lemma \ref{l:everyone_is_a_small_extension}. So we assume that this sequence has been constructed, and focus on the definition of $j_n$; assume that all our conditions are satisfied up to rank $n-1$ and we have to define $j_n$.

Given $i \in \{1,\ldots,k_n\}$, let $\sigma_i$ be the element of $\Lambda_n$ mapping $\alpha_i^n$ to $\beta_i^n$. Defining $j_n$ amounts to defining $j_n(\sigma_i)$ for all $i$; we first deal with $\sigma_1$.
Let $\Omega$ be a $\mcB_n$-orbit. 

If $\Omega$  not an exceptional orbit, we may list all atoms of $\Omega$ contained in $\alpha_1^n$ as $\{U_1,\ldots,U_p\}$ and all atoms of $\Omega$ contained in $\beta_1^n$ as $\{V_1,\ldots,V_p\}$ with $U_i$, $V_i$ belonging to the same $\mcA_{n+1}$-orbit for all $i$, then set $j_n(\sigma_1)(U_k)=V_k$. Note that on $\Omega$ $j_n(\sigma_1)$ coincides with an element of $\tilde \Gamma_{n+1}$.

Else, $\Omega$ is the disjoint union of the $\mcA_{n+1}$-orbit $\Omega_1$ of some $\alpha_j^{n+1}$ and the $\mcA_{n+1}$-orbit $\Omega_2$ of $\beta_j^{n+1}$. Note that $\alpha_j^{n+1} \subseteq \alpha_1^n$ and $\beta_j^{n+1} \subseteq \beta_1^n$ so we may set $j_n(\sigma_1)(\alpha_j^{n+1})=\beta_j^{n+1}$ and $j_n(\sigma_1)(\beta_j^{n+1})=\alpha_j^{n+1}$.

There exists $p \le N_1^{(n)} +1$ atoms $U_1,\ldots,U_p \sim_\Gamma \alpha_j^{n+1}$ and $V_1,\ldots,V_p \sim_\Gamma \beta_j^{n+1}$ such that $\Omega_1 \setminus \{U_1,\ldots,U_p\}$ and $\Omega_2 \setminus\{V_1,\ldots,V_p\}$ each have as many atoms contained in $\alpha_1^n$ and $\beta_ 1^n$ (the term ``$+1''$ comes from the fact that by asking that $\alpha_j^{n+1}$ be mapped to $\beta_j^{n+1}$ we may have slightly increased the imbalance between $\Omega_1$ and $\Omega_2$). We may then pair these other atoms to define $j_n(\sigma_1)$ so that $j_n(\sigma_1)$ sends each of them to an atom belonging to the same $\mcA_{n+1}$-orbit. We also agree that, for all $i \in \{1,\ldots,p\}$,
\[j_n(\sigma_1)(U_i)= V_i \quad \text{ and } j_n(\sigma_1)(V_i)= U_i .\] 

We have now defined $j_n(\sigma_1)$. There are at most $N_1^{(n)}+2$ atoms $U \in \mcA_n$  such that $j_n(\sigma_1)(U)$ and $U$ belong to different $\mcA_{n+1}$-orbits. By induction, we see that we have ensured $\pi_{n+1}(\alpha_j^{n+1})=\beta_j^{n+1}$ for each $j$, since we had $\pi_n(\alpha_1^n)=\beta_1^n$. 

To define $j_n(\sigma_i)$ for $i \ge 2$, we proceed similarly: if $\Omega$ is $(\alpha_i^n,\beta_i^n)$-balanced, then we match atoms of $\Omega$ as above so that on $\Omega$ $j_n(\sigma_i)$ coincides with an element of $\tilde \Gamma_{n+1}$.
Else, there must again exist some $j$ such that $\Omega$ is the disjoint union of the $\mcA_{n+1}$-orbit $\Omega_1$ of some $\alpha_j^{n+1}$ and the $\mcA_{n+1}$-orbit $\Omega_2$ of $\beta_j^{n+1}$, and we again find $U_1,\ldots,U_p \sim_\Gamma \alpha_j^{n+1}$, $V_1,\ldots,V_p \sim_\Gamma \beta_j^{n+1}$ such that $p \le N_j^{(n)}$ and $\Omega_1 \setminus \{U_1,\ldots,U_p\}$, $\Omega_2 \setminus\{ V_1,\ldots,V_p\}$ each have as many atoms contained in $\alpha_i^n$ and $\beta_i^n$. We can pair these atoms to define $j_n(\sigma_i)$ there (and on those atoms it coincides with an element of $\tilde \Gamma_n$). We again set $j_n(\sigma_i)(U_r)=V_r$, $j_n(\sigma_i)(V_r)=U_r$ for each $r$ in $\{1,\ldots,p\}$.

This completes the definition of $j_n$. 

Denote by $l_n$ the number of singular atoms in $\mcA_n$. Since $j_n$ coincides with $i_n$ on $\tilde \Gamma_n$, any singular atom $U$ of $\mcA_{n+1}$ is contained in a singular atom $U'$ of $\mcA_n$; and there exists some $\gamma_1,\gamma_2 \in \tilde \Gamma_n$ and $i \in \{1,\ldots,k_n\}$ such that $\pi_n(U)= \gamma_1 j_n(\sigma_i) \gamma_2(U)$. Given that fewer than $2(N^{(n)}+2)$ singular atoms have been created from the two singular atoms constituting the support of $\sigma_i$ when defining $j_n(\sigma_i)$, a very coarse estimate gives: 
\[l_{n+1} \le l_n (N^{(n)}+2) \le h_n (N^{(n)}+2) .\]
Since there are more than $(n+1) h_n (N^{(n)}+2)$ fragments of each $\mcA_n$-orbit in every exceptional $\mcA_{n+1}$-orbit, $\mu(K_{n+1}) \le \frac{1}{n+1}$ for all $\mu \in M(\tilde \Gamma)$.
\end{proof}

With these definitions in hand, let $\tilde \Gamma$ be the inductive limit of $(\tilde \Gamma_n,i_n)$ and $\Lambda$ be the inductive limit of $(\Lambda_n,j_n)$. Denote also by $j_n$ the embedding of $\Lambda_n$ in $\Lambda$ and let $\pi=j_0(\pi_0)$. Since $\pi(\alpha_j^n)=\beta_j^n$ for all $j$ and all $n$, $\Lambda$ is generated as a full group by $\tilde \Gamma$ and $\pi$. 

Then $\tilde \Gamma$ is a minimal ample group and $\sim_{\tilde \Gamma}$ coincides with $\sim_\Gamma$; $\Lambda$ is a saturated minimal ample group and $\sim_\Lambda$ coincides with $\sim_\Gamma^*$ (this is straightforward to check, for details see the arguments of \cite{Melleray2023}). We thus have $M(\Lambda)=M(\tilde \Gamma)=M(\Gamma)$.

Let $K=\bigcap_n K_n$. It is closed, and $\mu(K)=0$ for every $\mu \in M(\Lambda)$.

We now show that $K$ is $\Lambda$-étale.
Let $U$ be clopen in $K$ and $\gamma \in \Lambda$. Then for some $n$ we have both that $\lambda \in j_n(\Lambda_n)$ and that there exist singular atoms $U_1,\ldots,U_p$ of $\mcA_n$ such that $U \cap K= \bigsqcup_i (U_i \cap K)$. We may as well assume that $\lambda$ is the involution of $j(\Lambda_n)$ with support $U_i \sqcup \lambda(U_i)$ for some $i \in \{1,\ldots,p\}$. If $\lambda(U_i)$ is not a singular atom then $\lambda (U_i) \cap K= \emptyset$. Assume now that $\lambda(U_i)$ is singular; there are two cases to consider. If $\lambda(U_i)$ and $U_i$ belong to the same $\mcA_n$-orbit then there exists $\gamma \in \Gamma'$ such that $\lambda$ coincides with $\gamma$ on $U_i$, whence $\lambda(U_i \cap K)=\lambda(U_i) \cap K$.
Else, there exist $\gamma_1, \gamma_2 \in \Gamma'$ and some $j$ such that $\lambda$ coincides on $U_i$ with $\gamma_1 \pi \gamma_2$, and $\gamma_2(U_i)$ is singular. We again obtain  $\lambda(U_i \cap K) = \lambda(U_i) \cap K$. This proves that $K$ is $\Lambda$-malleable (hence also $\tilde \Gamma$-malleable).

For any $x \not \in K$ we have $\pi(x) \in \tilde \Gamma x$. Since $\Lambda$ is generated, as a full group, by $\tilde \Gamma$ and $\pi$, this implies that $R_\Lambda=R_{\tilde \Gamma}(K,\Lambda_{|K})$.

We finally conclude, thanks to Theorem \ref{t:absorption}, that $R_\Gamma$ is orbit equivalent to $R_\Lambda$; so every minimal ample group is orbit equivalent to a saturated minimal ample group and (by Krieger's theorem) this concludes the proof of the classification theorem for minimal ample groups.

\bigskip

Writing this appendix offers an opportunity to discuss another imprecision in \cite{Melleray2023}. In part (2) of \cite{Melleray2023}*{Lemma 3.4}, the assumption that $\Gamma$ acts topologically transitively is not sufficient for the argument to go through (despite our claim that ``it is the natural hypothesis to make the argument work''...). The reason is that the proof uses implicitly that for any nonempty clopen set $A$ one has $\inf_{\mu \in M(\Gamma)} \mu(A) >0$. That condition is in fact equivalent to assuming that $\Gamma$ acts minimally on $X$. I do not know in which generality point (2) of \cite{Melleray2023}*{Lemma~3.4} holds. Fortunately, this does not affect the arguments of \cite{Melleray2023} that use Lemma~3.4 since we are everywhere concerned with minimal actions.
    \end{appendices}

\bibliography{biblio_GPS}

@ARTICLE{Giordano2004,
  author = {Giordano, Thierry and Putnam, Ian F. and Skau, Christian F.},
  title = {Affable equivalence relations and orbit structure of {C}antor dynamical
	systems},
  journal = {Ergodic Theory Dynam. Systems},
  year = {2004},
  volume = {24},
  pages = {441--475},
  number = {2},
  fjournal = {Ergodic Theory and Dynamical Systems},
  issn = {0143-3857}
  }

@ARTICLE{Giordano1995,
  author = {Giordano, Thierry and Putnam, Ian F. and Skau, Christian F.},
  title = {Topological orbit equivalence and {$C^*$}-crossed products},
  journal = {J. Reine Angew. Math.},
  year = {1995},
  volume = {469},
  pages = {51--111},
  coden = {JRMAA8},
  fjournal = {Journal f\"ur die Reine und Angewandte Mathematik},
  issn = {0075-4102}
 }

@article {Giordano2008,
    AUTHOR = {Giordano, Thierry and Matui, Hiroki and Putnam, Ian F. and
              Skau, Christian F.},
     TITLE = {The absorption theorem for affable equivalence relations},
   JOURNAL = {Ergodic Theory Dynam. Systems},
  FJOURNAL = {Ergodic Theory and Dynamical Systems},
    VOLUME = {28},
      YEAR = {2008},
    NUMBER = {5},
     PAGES = {1509--1531},
      ISSN = {0143-3857},
}

@article{Giordano2010,
  title = {Orbit Equivalence for {{Cantor}} Minimal {{Z\textsuperscript{d}-systems}}},
  author = {Giordano, Thierry and Matui, Hiroki and Putnam, Ian F. and Skau, Christian F.},
  year = {2010},
  journal = {Inventiones Mathematicae},
  volume = {179},
  number = {1},
  pages = {119--158},
  issn = {0020-9910},
  doi = {10.1007/s00222-009-0213-7},
  coden = {INVMBH},
  fjournal = {Inventiones Mathematicae},
  mrclass = {37B05 (54H20)},
  mrnumber = {2563761 (2011d:37013)},
  mrreviewer = {M{\'a}rio Bessa}
}

@ARTICLE{Glasner1995a,
  author = {Glasner, Eli and Weiss, Benjamin},
  title = {Weak orbit equivalence of {C}antor minimal systems},
  journal = {Internat. J. Math.},
  year = {1995},
  volume = {6},
  pages = {559--579},
  number = {4},
  fjournal = {International Journal of Mathematics},
  issn = {0129-167X},
 }

@article{Krieger1979,
    AUTHOR = {Krieger, Wolfgang},
     TITLE = {On a dimension for a class of homeomorphism groups},
   JOURNAL = {Math. Ann.},
  FJOURNAL = {Mathematische Annalen},
    VOLUME = {252},
      YEAR = {1979/80},
    NUMBER = {2},
     PAGES = {87--95},
      ISSN = {0025-5831},
  }

@article{Keane2011,
    AUTHOR = {Hamachi, Toshihiro and Keane, Michael S. and Yuasa, Hisatoshi},
     TITLE = {Universally measure-preserving homeomorphisms of {C}antor
              minimal systems},
   JOURNAL = {J. Anal. Math.},
  FJOURNAL = {Journal d'Analyse Math\'{e}matique},
    VOLUME = {113},
      YEAR = {2011},
     PAGES = {1--51},
      ISSN = {0021-7670},
}

@book{Putnam2018,
    AUTHOR = {Putnam, Ian F.},
     TITLE = {Cantor minimal systems},
    SERIES = {University Lecture Series},
    VOLUME = {70},
 PUBLISHER = {American Mathematical Society, Providence, RI},
      YEAR = {2018},
     PAGES = {xiii+149},
      ISBN = {978-1-4704-4115-9}
}

@article {Matui2008,
    AUTHOR = {Matui, Hiroki},
     TITLE = {An absorption theorem for minimal {AF} equivalence relations
              on {C}antor sets},
   JOURNAL = {J. Math. Soc. Japan},
  FJOURNAL = {Journal of the Mathematical Society of Japan},
    VOLUME = {60},
      YEAR = {2008},
    NUMBER = {4},
     PAGES = {1171--1185},
      ISSN = {0025-5645},
}

@article {Putnam2010,
    AUTHOR = {Putnam, Ian F.},
     TITLE = {Orbit equivalence of {C}antor minimal systems: a survey and a
              new proof},
   JOURNAL = {Expo. Math.},
  FJOURNAL = {Expositiones Mathematicae},
    VOLUME = {28},
      YEAR = {2010},
    NUMBER = {2},
     PAGES = {101--131},
      ISSN = {0723-0869}
}

@article {Melleray2023,
    AUTHOR = {Melleray, Julien and Robert, Simon},
     TITLE = {From invariant measures to orbit equivalence, via locally
              finite groups},
   JOURNAL = {Ann. H. Lebesgue},
  FJOURNAL = {Annales Henri Lebesgue},
    VOLUME = {6},
      YEAR = {2023},
     PAGES = {259--295},
      ISSN = {2644-9463},
   MRCLASS = {37A20 (37B05)},
  MRNUMBER = {4648084},
       DOI = {10.5802/ahl.165},
       URL = {https://doi.org/10.5802/ahl.165},
}

\end{document}